\documentclass[italian,a4paper,10pt]{article}
\usepackage[latin1]{inputenc}
\usepackage[english]{babel}
\usepackage{babel,amsmath,amssymb,amsbsy,amsfonts,latexsym,amsthm,mathrsfs}
\usepackage{xcolor}
\usepackage{version}

\usepackage{amsmath,amsfonts,latexsym}
\newtheorem{theorem}{Theorem}[section]
\newtheorem{proposition}[theorem]{Proposition}
\newtheorem{definition}[theorem]{Definition}
\newtheorem{lemma}[theorem]{Lemma}

\newtheorem{remark}[theorem]{Remark}
\newtheorem{example}[theorem]{Example}
\newtheorem{assumption}[theorem]{Assumption}

\def\proof{{\sl Proof. }}
\def\proofof#1{{\sl Proof of #1.}}
\def\cl#1{{\mathscr #1}}

\def\P{{\mathbb P}}
\def\R{{\mathbb R}}

\def\N{{\mathbb  N}}

\def\E{{\mathbb E}}

\def\L{{\mathbb L}}

\def\Var{{\rm Var }}
\def\Cov{{\rm Cov }}

\def\<{\langle}
\def\>{\rangle}

\newcommand{\cvd}{
                   $\quad\Box $
                   \medskip
                   }

\def\appendix{\par
  \setcounter{chapter}{0}
  \def\@chapter{Appendix}
  \def\thechapter{\Alph{chapter}}}
\font\tengoth=eufm10 scaled 1200 \font\sevengoth=eufm7 scaled 1200
\font\fivegoth=eufm5 scaled 1200
\newfam\gothfam
\textfont\gothfam=\tengoth \scriptfont\gothfam=\sevengoth
\scriptscriptfont\gothfam=\fivegoth

\def\ep{\varepsilon}

\begin{document}
\title{
\bf
Pathwise asymptotics for Volterra type stochastic volatility models}
\author{
   { {\sc Miriana Cellupica} } \and {\sc Barbara Pacchiarotti}\thanks{Dept. of Mathematics, University of Rome ``Tor Vergata'', E-mail address: \texttt{pacchiar@mat.uniroma2.it}} }

\date{}
\maketitle

{\small
\bf Abstract. \rm We study  stochastic volatility models in which the volatility process
is a positive continuous function of a continuous  Volterra stochastic process.
We state some pathwise
 large deviation  principles for the scaled log-price.
}

\bigskip


\bigskip


\bigskip

{\small
\bf Corresponding Author\rm: Barbara Pacchiarotti,
Dipartimento di Matematica,
Universit\`a di Roma \sl Tor Vergata\rm, Via della Ricerca
Scientifica, I-00133 Roma, Italy. E-mail address: \texttt{pacchiar@mat.uniroma2.it}
}

\section{Introduction}\label{sect:intro}
The last few years have seen renewed interest in stochastic volatility
models driven by fractional Brownian motion or other self-similar Gaussian processes
(see \cite{GJR}, \cite{GVZ}, \cite{GVZ2}), i.e. fractional stochastic volatility models.
Here we study  stochastic volatility models in which the volatility process
is a positive continuous function $\sigma$ of a continuous  stochastic process $\hat B$, that we assume to be  a Volterra type Gaussian process. The
main result obtained in the present paper
is  a generalization, to the infinite-dimensional case, of a large deviation  principle
for the log-price process, due (in the real case) to Forde and Zhang \cite{For-Zha} and Gulisashvili \cite{Gu1}.
Our result is similar to that obtained in Section 2 in  a recent paper of Gulisashvili \cite{Gu2} where also moderate deviations are considered.

An important aspect of this paper is that the techniques  used here are  different from those  generally used in this framework (Freidlin-Wentzell theory).
The principal result we use is Chaganty Theorem (see Theorem 2.3 in \cite{Cha}), where a large deviation  principle for joint and marginal  distributions is stated.
In this way  the same results contained in \cite{For-Zha}, \cite{Gu1} and \cite{Gu2} can  be obtained in a more general context, see Section \ref{sect:VolterraGen}.

In
the stochastic volatility models of interest, the dynamic of the asset price process $ (S_t)_{t \in [0,T]} $ is modeled by the following equation

$$
\begin{cases}
dS_t = S_t\, \mu(\hat{B}_t)dt + S_t\,\sigma(\hat{B}_t)d(\bar{\rho}W_t+\rho B_t), \qquad 0\leq t \leq T,\\
S_0=s_0 >0,
\end{cases}
$$
where  $s_0$ is the initial price, $ T > 0 $ is the time horizon and
$ \hat{B} $ is a non-degenerate continuous Volterra type Gaussian process of the form
$$ \hat{B}_t= \displaystyle\int_0^t K(t,s)\, dB_s, \quad 0\leq t \leq T, $$
for some kernel $ K$,
 the processes $ W $ and $ B $  are two independent standard Brownian motions, $ \rho \in (-1,1) $ is the correlation coefficient and $ \bar{\rho}=\sqrt{1-\rho^2}.$ Remark that $ \bar{\rho}W+\rho B $ is another standard Brownian motion which has correlation coefficient $ \rho $ with $ B $. If $ \rho \neq 0 $ the model is called a correlated stochastic volatility model, otherwise it is called an uncorrelated model.
It is assumed that $ \mu: \mathbb{R} \to \mathbb{R} $ and $ \sigma: \mathbb{R} \to (0, +\infty) $ are continuous functions satisfying suitable hypotheses. The process $ \sigma(\hat{B}) = (\sigma(\hat{B}_t))_{0\leq t \leq T } $ describes the stochastic evolution of the volatility in the model and $ \mu(\hat{B})=(\mu(\hat{B}_t))_{0\leq t \leq T} $ is an adapted return process. Note that the model here considered contains a drift term which is not present in \cite{Gu2}.

The unique solution to the previous equation  is the Dol\'{e}ans-Dade expression
$$
S_t=s_0 \exp \Big\{\displaystyle\int_0^t \mu(\hat{B}_s)\, ds-\frac 12 \int_0^t \sigma(\hat{B}_s)^2 \, ds +\bar{\rho} \int_0^t \sigma(\hat{B}_s)\, dW_s + \rho \int_0^t \sigma(\hat{B}_s)\, dB_s  \Big\}
$$
for $  0\leq t \leq T. $
Therefore, the log-price process $ Z_t= \log S_t$, $ 0\leq t \leq T, $ with $ Z_0=x_0=\log s_0 $ is defined by,
$$
Z_t=x_0+ \int_0^t\mu(\hat{B}_s)\,ds-\frac 12\int_0^t\sigma(\hat{B}_s)^2 ds +  \bar{\rho}\int_0^t\sigma(\hat{B}_s) dW_s + \rho\int_0^t \sigma(\hat{B}_s) dB_s.
$$

Now let $ \varepsilon_.: \mathbb{N} \to \mathbb{R}_+ $ be an infinitesimal function, i.e.    $ \varepsilon_n \to 0 $, as $ n \to +\infty$ .  For every $ n \in \mathbb{N} $, we  consider the following scaled version of the stochastic differential equation
$$\begin{array}{c}
\begin{cases}
dS_t^{n} =S_t^{n}\mu(\hat{B}_t^n)dt+ \varepsilon_n  S_t^{n}\sigma(\hat{B}_t^n)d(\bar{\rho}W_t+\rho B_t), \qquad 0\leq t \leq T,\\
S_0^{n}=s_0,
\end{cases}
\end{array}$$
Here  the Brownian motion $ \bar{\rho}W+\rho B $ is multiplied by a \it  small-noise \rm parameter $ \varepsilon_n$
and also the Volterra process $\hat B$  is  multiplied by a \it  small-noise \rm parameter, i.e.
$$\hat{B}^n_t=\varepsilon_n\hat{B}_t, \quad t \in [0,T].
$$

The  log-price process $ Z_t^{n}=\log S_t^{n}, $ $ 0\leq t \leq T,$
in the scaled model is
$$
Z_t^{n}=x_0+\displaystyle\int_0^t\mu(\hat{B}_s^n)\,ds-\frac 12\varepsilon_n^{2}\int_0^t\sigma(\hat{B}_s^n)^2 ds +  \varepsilon_n\,\bar{\rho}\int_0^t\sigma(\hat{B}_s^n) dW_s +  \varepsilon_n\,\rho\int_0^t\sigma(\hat{B}_s^n) dB_s.
$$
We will obtain a  sample path  large deviation  principle (which is called a small-noise large deviation  principle) for the family  of  processes $ ((Z_t^{n}-x_0)_{t \in [0,T]})_{n \in \mathbb{N}}.$
A large deviation  principle  for $ (Z_T^{n}-x_0)_{n \in \mathbb{N}} $ can be  obtained with the same techniques. But  it can be also obtained
by contraction and this is the approach we follow here.
The paper is organized as follows.

In Section \ref{sect:ldp} we recall some basic facts about large deviations for continuous Gaussian processes (also for Gaussian diffusions) and we give some examples of Volterra type stochastic volatility models to which the large deviation  principles obtained here could be applied. In particular we discuss fractional models.
In Section \ref{sect:Chaganty} we recall some basic facts about large deviations   for joint and marginal distributions.
In Sections \ref{sect:uncorrelatedSVM}, \ref{sect:correlatedSVM} and \ref{sect:VolterraGen} are contained the main results.
More precisely in Section \ref{sect:uncorrelatedSVM} we prove a large deviation  principle for the log-price process in  the uncorrelated model under mild  hypotheses on the coefficients $\mu$ and $\sigma$.
We assume only $\mu$ continuous,  $\sigma$ continuous and positive.
In \cite{Gu2} the uncorrelated model   is a particular case of the correlated one and it is obtained under the further hypothesis that $\sigma$ is locally $\omega$-continuous (see Section \ref{sect:correlatedSVM} for the exact definition).
In Section \ref{sect:correlatedSVM}, we
first prove a large deviation  principle for a certain family
$ (Z^{n,m}-x_0)_{n \in \mathbb{N}} $  with a certain good rate function $ I^m$ (Section \ref{sect:appr-fam}). Then, showing that the family $ ((Z^{n,m})_{n \in \mathbb{N}})_{m\in \N} $ is an exponentially good approximation of $ (Z^{n})_{n \in \mathbb{N}} $, we  prove a large deviation  principle for  $ (Z^{n}-x_0)_{n \in \mathbb{N}} $ with the good rate function obtained in terms of the $ I^m $'s (Section \ref{sect:correlatedSVM-notid}). To prove this large deviation
principle we have the same hypotheses on $\sigma$ as in  \cite{Gu2}.
Finally, in Section \ref{sect:id-rate}, we give an explicit expression for the rate function (not in terms of the $ I^m $'s).
 In Section \ref{sect:id-rate}, for the identification of the rate function,  we have a more restrictive hypothesis on $\sigma$. Here we need a   power growth not required in \cite{Gu2}.
In Section \ref{sect:cross-prob} we give an application to the asymptotic estimate of the crossing probability.
In  Section \ref{sect:VolterraGen} we extend the  results of  Sections
\ref{sect:uncorrelatedSVM} and \ref{sect:correlatedSVM} to a    more general context.
We get the  same results   for    more general families of Volterra processes $((\hat{B}^n_t)_{t \in [0,T]})_{n\in \N}$ (not only  $((\varepsilon_n\hat{B}_t)_{t \in [0,T]})_{n\in \N}$) that obey
a large deviation  principle (some examples of such processes can be found in \cite{Car-Pac-Sal}, \cite{Gio-Pac} and \cite{Pac}). For example we can consider  $((\hat B^n_{t \in [0,T]})_{n\in \N}) =((\hat B_{\ep_n t})_{t \in [0,T]})_{n\in \N}$(Example \ref{ex:not-ss}). In this case we have a   small-time
large deviation  principle for the Volterra processes.
If the Volterra process is self-similar we can pass  from small-noise to small-time regime (see the discussion at the end of Section 3 in \cite{Gu2}), while, if the process is not self-similar, it is not generally true.

\section{Large deviations for continuous Gaussian processes} \label{sect:ldp}

We briefly recall some main facts on large deviation  principles and Volterra processes   we are going to use.
For a detailed development of this very wide theory we can refer, for example, to the following classical references:
 Chapitre II in Azencott \cite{Aze},
Section 3.4 in Deuschel and Strook \cite{Deu-Str}, Chapter 4 (in particular Sections 4.1, 4.2  and 4.5) in  Dembo and Zeitouni
\cite{Dem-Zei}, for large deviation  principles;   \cite{De-Ust} and  \cite{Hult} for
Volterra processes.

\subsection{Large deviations}

\begin{definition}
Let $E$ be a topological space,   $\cl{B}(E)$ the Borel $\sigma$-algebra and  $(\mu_n)_{n\in \N}$ a family of probability measures on $\cl{B}(E)$; let $\gamma \, : \N \rightarrow \mathbb{R}^+ $ be a speed function, i.e. $\gamma_n \rightarrow +\infty$ as $n\to +\infty$. We say that the family of probability measures $(\mu_n)_{n\in \N}$ satisfies a large deviation  principle (LDP) on $E$ with the rate function $I$ and the speed $\gamma_n$ if, for any open set
$\Theta$,
$$
-\inf_{x \in {\Theta} } I(x) \le \liminf_{n\to+\infty} \frac 1{\gamma_n} \log \mu_n (\Theta)
$$ and for any closed set $\Gamma$
\begin{equation} \label{eq:upperbound} \limsup_{n\to +\infty}\frac 1 {\gamma_n} \log \mu_n (\Gamma) \le -\inf_{x \in {\Gamma}} I(x).
\end{equation}
\end{definition}
A rate function is a lower semicontinuous mapping $I:E\rightarrow [0,+\infty]$. A rate function $I$ is said \textit{good} if  $\{I\le a\}$ is a compact set for every $a \ge 0$.
\begin{definition} \label{def:wldp}
Let $E$ be a topological space,   $\cl{B}(E)$ the Borel $\sigma$-algebra and  $(\mu_n)_{n\in \N}$ a family of probability measures on $\cl{B}(E)$; let $\gamma \, : \N \rightarrow \mathbb{R}^+ $ be a speed function. We say that the family of probability measures $(\mu_n)_{n\in \N}$ satisfies a weak large deviation  principle (WLDP) on $E$ with the rate function $I$ and the speed $\gamma_n$ if the upper bound (\ref{eq:upperbound}) holds for compact sets.
\end{definition}

Let $U=(U_t)_{t\in[0,T]}$ be a continuous, centered, Gaussian process on a probability space
$(\Omega,\cl F, \P)$. From now on, we will denote by  $C [0, T]$ (respectively $ C_0[0,T]$) the set of
continuous functions on $[0, T]$ (respectively the set of
continuous functions on $[0, T]$ starting from 0) endowed with the topology induced by the sup-norm.
Moreover,
we will denote by $\cl M[0, T]$ its dual, that is, the set of signed Borel measures on $[0, T]$.
The action of $\cl M[0, T]$ on $C [0, T]$ is given by
$$\langle \lambda,h \rangle=\int_0^T h(t) d\lambda(t), \quad \lambda\in\cl M[0, T], \,h\in C [0, T].$$

\begin{remark}\rm
We say that a family of continuous processes $((U_t^n)_{t \in [0,T]})_{n\in \N}$, $U^n_0=0$ satisfies a LDP if  the family of their laws satisfies a LDP on $ C_0[0,T]$.
\end{remark}

The following remarkable theorem (Proposition 1.5 in \cite{Aze})  gives an explicit expression of the Cram\'er transform $\Lambda^*$ of a continuous centered Gaussian process $(U_t)_{t \in [0,T]}$   with covariance function $k$. Let us recall that
$$\Lambda(\lambda )=\log \E[\exp(\langle U, \lambda\rangle )]=\frac{1}{2} \int_0^T \int_0^T k(t,s) \, d\lambda(t) d\lambda(s)$$ for $\lambda\in \cl M[0,T]$.

\begin{theorem} \label{th:cramer-transform}
Let $(U_t)_{t \in [0,T]}$ be a continuous, centered Gaussian process, with covariance function $k$. Let $\Lambda^*$ denote the Cram\'er transform of  $\Lambda$, that is
$$
\Lambda^*(x) = \sup_{\lambda \in \cl M[0,T]} \left( \langle \lambda, x \rangle - \Lambda(\lambda) \right)
= \sup_{\lambda \in \cl M[0,T]} \left( \langle \lambda, x \rangle -  \frac{1}{2} \int_0^T \int_0^T k(t,s) \, d\lambda(t) d\lambda(s) \right).
$$
Then,
$$
\Lambda^*(x) = \begin{cases} \frac{1}{2} \|x \|_\cl{H}^2 & x \in \cl{H} \\
					      +\infty					    & x \notin \cl{H}.
			\end{cases}
$$
where $\cl{H}$ and $\| . \|_\cl{H}$ denote, respectively, the reproducing kernel Hilbert space and the related norm associated to the covariance function $k$.
\end{theorem}
Reproducing kernel Hilbert spaces (RKHS) are an important tool to handle Gaussian processes. For a detailed development of this wide
theory we can refer, for example, to  Chapter 4 in \cite{Hid-Hit} (in particular Section 4.3) and to  Chapter 2  (in particular Sections 2.2 and 2.3) in \cite{Ber-Tho}.
In order to state a large deviation principle for a family of Gaussian processes, we need the following definition.
\begin{definition}
A family of continuous processes ${((X^n_t)_{t \in [0,T]}})_{n\in\N}$ is exponentially tight at  the speed $\gamma_n$ if, for every $R>0$ there exists a compact set $K_R$ such that
$$
\limsup_{n\to +\infty}{\gamma_n^{-1}} \log \P (X^n \notin K_R) \le -R.
$$
\end{definition}

If the means and the covariance functions of an exponentially tight family of Gaussian processes have a good limit behavior, then the family satisfies a large deviation principle, as stated in the following theorem
which is a consequence of the classic abstract G$\ddot{\rm a}$rtner-Ellis Theorem (Baldi Theorem 4.5.20 and Corollary 4.6.14 in \cite{Dem-Zei}) and Theorem \ref{th:cramer-transform}.
\begin{theorem}\label{th:ldp-gaussian}
Let $((X_t^n)_{t \in [0,T]})_{n\in\N}$ be an exponentially tight family of continuous Gaussian processes with respect to the speed function $\gamma_n$. Suppose that, for any $\lambda \in \cl M[0,T]$,
$$
\lim_{n\to +\infty} \E\left[ \langle \lambda , X^n \rangle \right] = 0
$$
and the limit
\begin{equation}\label{eq:covlimit}
\Lambda(\lambda) = \lim_{n\to +\infty} \gamma_n\Var\left( \langle \lambda , X^n \rangle \right) = \int_0^T \int_0^T {k}(t,s) \, d\lambda(t) d\lambda(s)
\end{equation}
exists, for some continuous, symmetric, positive definite function ${k}$, that is  the covariance function of a continuous Gaussian process,
then $((X_t^n)_{t \in [0,T]})_{n\in\N}$
 satisfies a large deviation principle on $C[0,T]$, with  the speed $\gamma_n$ and the good rate function
$$
I(h) = \begin{cases} \frac{1}{2} \left \| h \right \|^2_{{\cl  H }} & h \in{\cl  H }\\ +\infty & x \notin \cl{H}, \end{cases}$$
where ${{\cl H}}$ and $\left \| . \right \|_{{\cl  H }}$ denote,  respectively, the reproducing kernel Hilbert space and the related norm associated to the covariance function ${k}$.
\end{theorem}

\begin{remark}\label{exp-equiv}
\rm Suppose $ ((U_t^n)_{t \in [0,T]})_{n\in \N} $ is a family of centered Gaussian processes  that satisfies a large deviation principle on $C[0,T]$ with the speed $\gamma_n $ and the  good rate function $I $. Let $(m^n)_{n\in \N}\subset C[0,T]$,   $m\in C[0,T] $ be functions such that
	$ m^n\overset{C[0,T]}{\underset{}{\longrightarrow}} {m}, $  as $ n \to +\infty. $
	Then, the family of processes $ (X^n)_{n\in \N}$, where $X^n=m^n +U^n$, satisfies a large deviation principle on $C[0,T] $ with the same  speed $\gamma_n $ and  the good rate function
	$$I_X(h) =I(h-m)= \begin{cases} \frac{1}{2} \left \|h-m \right \|^2_{{\cl  H }} & h- m \in{\cl  H }\\ +\infty & h- m \notin{\cl H }. \end{cases}$$
\end{remark}
 A useful result which can help in investigating the exponential tightness of a family of continuous  Gaussian processes is Proposition 2.1 in \cite{Mac-Pac} where   the required property follows from H\"older continuity of the mean and the covariance function.

\subsection{Volterra type Gaussian processes}

Let $ (\Omega,\cl{F},\mathbb{P}) $ be a probability space and $ B=(B_t)_{t \in [0,T]} $ a standard Brownian motion.
Suppose $ \hat{B}=(\hat{B}_t)_{t \in [0,T]} $ is a centered Gaussian process having the following Fredholm representation,
\begin{eqnarray}\label{eq:integral representation} \hat{B}_t= \int_0^T K(t,s)\, dB_s, \quad 0 \leq t \leq T,
\end{eqnarray}
where  $ T > 0 $ and  $ K $ is a  measurable square integrable kernel on $ [0,T]^2 $ such that
$$\sup_{t\in[0,T]} \int_0^T K(t,s)^2 \,ds < \infty. $$
For such a kernel, the linear operator $ \mathcal{K}:\L^2[0,T]\longmapsto \L^2[0,T] $ defined by
$$\mathcal{K}h(t)=  \int_0^T K(t,s)h(s)\, ds,$$
 is compact. The operator $ \mathcal{K} $ is called a Hilbert-Schmidt integral operator.
The modulus of continuity of the kernel $ K $  is defined as follows
$$
M(\delta)=  \sup_{\{t_1,t_2 \in [0,T]: |t_1-t_2|\leq \delta\}}  \int_0^T |K(t_1,s)-K(t_2,s)|^2\,ds, \quad 0 \leq \delta \leq T.
$$

  The covariance function of the process $ \hat{B} $ is given by
$$k(t,s)= \int_0^T K(t,u)K(s,u)\,du, \quad t,s \in [0,T].$$
Let us define a Volterra process.	
\begin{definition}\label{def:Volterra process}
	The process in (\ref{eq:integral representation}) is called a Volterra type Gaussian process if the following conditions hold for the kernel $ K $:
	\begin{enumerate}
		\item [\textit{(a)}] $K(0,s)=0$ for all $ 0\leq s \leq T, $ and $ K(t,s)=0 $ for all $ 0\leq t < s \leq T $;
		\item [\textit{(b)}] There exist constants $ c > 0 $ and $ \alpha > 0 $ such that $ M(\delta)\leq c\,\delta^\alpha $ for all $ \delta \in [0,T] $.
		\end{enumerate}
\end{definition}

\begin{remark}\rm
Condition \textit{(a)} is a typical Volterra type condition for the kernel $ K $ and the integral representation in (\ref{eq:integral representation}) becomes $\hat{B}_t= \int_0^t K(t,s)\, dB_s,$ for $ 0 \leq t \leq T. $ So $ \hat{B} $ is adapted to the natural filtration generated by $B$.
		Condition \textit{(b)} guarantees the existence of a H\"{o}lder continuous modification of the process $ \hat{B} $.
Note that other definitions for Volterra processes are allowed. For example,   Definition  5 in  \cite{Hult} also contains the following
condition
\begin{enumerate}
		\item [\textit{(c)}]    $ \mathcal{K} $ is injective as a transformation of functions in $ \L^2[0,T] $.
	\end{enumerate}
\end{remark}
Thanks to Condition \textit{(c)} an explicit expression for the RKHS holds, see the next Remark. We will not use condition \textit{(c)} in this paper.

\begin{remark}\label{rem:volterra kernel}\rm
		If $ (\hat B_t)_{t\in [0,T]} $ is a Volterra type Gaussian process with kernel $ K,$ satisfying condition \textit{(c)}, the reproducing kernel Hilbert space $ \mathscr{H}_{\hat B} $ can  be represented as the image of $ \L^2[0,T] $ under the integral transform $ \mathcal{K} $, i.e. $ \cl{H}_{\hat B}=\mathcal{K}(\L^2[0,T]) $, equipped with the inner product
		$$ \langle \varphi,\psi\rangle_{\cl{H}_{\hat B}}=\langle\mathcal{K}^{-1}\varphi, \mathcal{K}^{-1}\psi\rangle_{\L^2[0,T]}, \quad \varphi,\psi \in \cl{H}_{\hat B}, $$ (for further details, see e.g. Subsection 2.2 in \cite{Hult} and
\cite{Sot-Vii}).
		Any $ \varphi \in \mathscr{H}_{\hat B} $ can be represented as
		$$ \varphi(t)=\mathcal{K}{f}(t)= \int_0^tK(t,s){f}(s)\, ds, \quad t \in [0,T]$$
		where $ {f} $ belongs to $ \L^2[0,T].$ If condition \textit{(c)} is verified
		we have an identification between $ \varphi \in \mathscr{H}_{\hat B} $ and $ {f} \in \L^2[0,T]$ ($ \mathcal{K} $ is a bijection from $ \L^2[0,T] $ into $ \mathscr{H}_{\hat B} $).

\end{remark}

We now discuss some   Volterra processes which satisfy conditions \textit{(a)} and \textit{(b)} in  Definition \ref{def:Volterra process}.

 \textit{\textbf{Fractional Brownian motion}}. \rm The fractional Brownian motion $Z$ with Hurst parameter $H \in (0,1)$ is the centered Gaussian process with covariance function
\[
k_H(t,s)=\frac{1}{2}\left(t^{2H}+s^{2H}-|t-s|^{2H}\right).
\]
It is well-known that  fractional Brownian motion can be represented  as a  Volterra process with kernel
$$
\begin{array}{c} \label{eqn:kernel fbm}
K_H(t,s) = c_H \left[ \left( \frac{t}{s} \right) ^{H-1/2}(t-s)^{H-1/2} - \left( H-\frac{1}{2} \right) s^{1/2 - H} \int_s^t \!u^{H-3/2}(u-s)^{H-1/2} du \right],
\end{array}
$$
where
\[
c_H = \Big(  \frac{2H \, \Gamma(3/2-H)}{\Gamma(H+1/2) \, \Gamma(2-2H)}\Big)^{1/2}.
\]
\\
Notice that when $H=1/2$  the fractional Brownian motion reduces to the Wiener process.  Condition \textit{(b)} for this process, with $\alpha=\min\{2H,1\}$, was established in \cite{Zha} and Lemma 8 in \cite{Gu1}.

\textit{\textbf{Fractional Ornstein-Uhlenbeck process.}} \rm
For $H\in(0,1)$ and $a > 0$, the fractional Ornstein-Uhlenbeck process is given by
$$
U^H_t= \int_0^t e^{-a(t-u)} dB^H_u, \quad t\geq 0$$
where $B^H$ is a fractional Brownian motion and  the stochastic integral appearing above  can be defined using the integration
by parts formula and the stochastic Fubini theorem. This gives the following equality,
\begin{equation}\label{eq:FOU}
U^H_t=B^H_t - a\int_0^t e^{-a(t-u)}B^H_u du,\end{equation}
and therefore  the Volterra representation,
$$U^H_t=\int_0^t \tilde K_H(t,s) dB_s,$$
where, for $0\leq s \leq t \leq T$,
$$
\tilde K_H(t,s)= K_H(t,s)- a \int_s^t e^{-a(t-u)} K_H(t,u) du,$$
(see, e.g., Proposition A.1 in \cite{Ch-ka-Ma}). Condition \textit{(b)} for this process, with $\alpha=\min\{2H,1\}$, was established in Lemma 10 in \cite{Gu1}.
Note that this is not a self similar process, therefore large deviations for small-time cannot be deduced from large deviations for small-noise. See  Example \ref{ex:not-ss}.

\textit{\textbf{Riemann-Liouville fractional Brownian motion.}} \rm
For $H\in(0,1)$, the Riemann-Liouville fractional
Brownian motion is defined by
$$R^H_t= \frac 1{\Gamma(H+1/2)} \int_0^t (t-u)^{H-1/2} dB_s, \quad t\geq 0.$$
This process  is simpler than
fractional Brownian motion. However, the increments of the Riemann-Liouville fractional
Brownian motion lack the stationarity property. Condition \textit{(b)} for this process, with $\alpha=2H$, was established in Lemma 7 in \cite{Gu1}.

\textit{\textbf{$a$-th fold integrated Brownian motion.}}
 \rm For $a\in \N$, consider the Volterra process
$$
Z_t = \int_0^t \frac{(t-s)^a}{a!} \, dB_s, \quad t\geq 0.
$$

The covariance function is
\[
k(s,t) = \int_0^{s \land t} \frac{(t-u)^a (s-u)^a}{(a!)^2} \, du.
\]

This is the covariance function of the $a$-th fold integrated Brownian motion. For details, see \cite{Che-Li}. Notice that for this process the kernel is very similar to the kernel of  the Riemann-Liouville fractional Brownian motion. Condition \textit{(b)} for this process, with $\alpha=2a+1$, trivially holds.

\textit{\textbf{Conditioned Volterra processes.}}
For $T>0$ consider the centered Volterra process $\hat Z^T$
defined by
$$
\hat Z^T_t = \int_0^t K(T+t,T+u) \, dB_u, \quad t\geq 0.
$$
The
 covariance function is
 $$
\hat{k}^T(t,s)=\int_0^{s \land t} K(T+t,T+u)K(T+s,T+u) \, du.
$$
This process can be   obtained by conditioning a Volterra process with kernel $K$  to the past up to time $T$.
The new kernel is $\hat K^T(t,s)=K(T+t,T+s)$.
For major details see \cite{Gio-Pac}.
If the original kernel satisfies condition \textit{(b)} in $[0,2T]$ then  the new one satisfies condition \textit{(b)} in $[0,T]$
and therefore
the large deviation  principles obtained in this paper can be applied.

\medskip

Now we recall   a small noise large deviation principle for the couple $ (\ep_n B,\ep_n \hat{B})_{n \in \mathbb{N}} $ (see, for example, \cite{Gu1}).
First observe that $ (B,\hat{B}) $ is a Gaussian process (for details see, for example, \cite{For-Zha}) and therefore the following theorem is an application of Theorem 3.4.5 in \cite{Deu-Str}.
From now on we denote by $H_0^1[0,T] $   the Cameron-Martin space, i.e. the set  of absolutely continuous functions $ f $   such that $ f(0)=0 $ and $ \dot{f}\in \L^2[0,T]. $

\begin{remark}\label{rem:RKHS-couple}\rm It is known that  reproducing kernel Hilbert space of the couple $ (B,\hat{B}) $ is the Hilbert space
\begin{equation}\label{eq:H2}\cl H_{(B,\hat B)}=\{(f,g) \in C_0[0,T]^2: f \in H_0^1[0,T], \, g(t)=\int_0^tK(t,u)\dot{f}(u)\,du, \quad 0\leq t \leq T\}.\end{equation}
equipped with the norm
$$\lVert(f,g)\rVert_{\cl H_{(B,\hat B)}}=\frac12  \int_0^T \dot{f}(s)^2 \, ds,$$
see, for example,  the discussion in Section 6 in \cite{Gu1}.
\end{remark}

\begin{theorem}\label{th:LDP couple}
$ ((\ep_nB, \ep_n\hat{B}))_{n \in \mathbb{N}} $ satisfies a large deviation principle on $C_0[0,T]^2$ with the speed $ \ep_n^{-2} $ and the good rate function
	\begin{eqnarray}\label{eq:couple rate function}
		I_{(B,\hat{B})} (f,g)=
		 \begin{cases}
		\displaystyle \frac12  \int_0^T \dot{f}(s)^2 \, ds & (f,g) \in \cl H_{(B,\hat B)}\\
		\displaystyle +\infty & (f,g)\in C_0[0,T]^2\setminus\cl H_{(B,\hat B)}
		\end{cases}
		\end{eqnarray}
where $H_{(B,\hat B)}$ is defined in (\ref{eq:H2}).
\end{theorem}

\begin{remark}\label{rem:LDP Volterra}\rm For $f\in\ H_0^1[0,T]$, define
\begin{equation}\label{eq:hat-f} \hat f(t)=\int_0^t K(t,u)\dot f (u) \, du \quad t\in[0,T].\end{equation}
Then, from Theorem \ref{th:LDP couple} and the contraction principle, the family $ (\ep_n \hat{B})_{n \in \mathbb{N}} $ satisfies a large deviation principle on $ C_0[0,T] $
with the speed $ \ep_n^{-2} $ and the good rate function
\begin{eqnarray}\label{eq:volterra rate function}
I_{\hat{B}}(g)=\inf\Big\{\frac12 \int_0^T \dot{f}(s)^2 \, ds  : \hat f=g, \,\, f\in H_0^1[0,T]\Big\},
\end{eqnarray}
with the understanding $ I_{\hat{B}}(g)= +\infty $ if the set
is empty.
\end{remark}

\subsection{Gaussian diffusion processes}
Let   $ X^n$ be the solution of the following stochastic differential equation
\begin{eqnarray}\label{eq:general W-F}
\begin{cases}
dX_t^n=b_n(t)dt + {\varepsilon_n}c_n(t)dW_t \quad 0\leq t \leq T\\
 X_0^n=x \in \mathbb{R}.
\end{cases}
\end{eqnarray}
This is a Gaussian diffusion process. As a simple application of Theorem 3.1 in \cite{Chi-Fis} we have the following result for Gaussian diffusion processes.

\begin{theorem}\label{th:general W-F}
Suppose that $b_n\to b$ and $c_n\to c$ in $C[0,T]$ then the family $ (X^n)_{n\in \N} $ of solutions to the SDE (\ref{eq:general W-F}) satisfies a LDP with the  speed
$\ep_n^{-2}$ and the good rate function
\begin{eqnarray}\label{eq:general rate function W-F}
I(f)=  \inf \Big\{ \frac 12\int_0^T \dot{g}(t)^2\, dt: \, x+\int_0^t b(s)\,ds + \int_0^t c(s )\dot{g}(s)\, ds=f(t), \, g \in H_0^1[0,T]\Big\}
\end{eqnarray}
with the understanding $ I(f)= +\infty $ if the set
is empty.	
\end{theorem}

\begin{remark}
	\rm In the non-degenerate case, that is, if $c\geq\underline c>0$
 then the rate function (\ref{eq:general rate function W-F})
 simplifies to
 $$I(f)= \begin{cases}\displaystyle\frac 12  \int_0^T \Big(\frac{\dot{f}(s)-{b}(s)}{c(s)}\Big)^2ds& f \in H^1_0[0,T]\\
\displaystyle +\infty & f \notin H^1_0[0,T].\end{cases}$$
 \end{remark}


\section{Large deviations for joint and  marginal distributions}\label{sect:Chaganty}
In this section we introduce the Chaganty Theorem in which a large deviation principle  for a sequence of probability measures on a product space $ E_1 \times E_2 $ (and then for both marginals)
is obtained
starting from the large deviation principle of  the  sequences of marginal and conditional distributions.
The main reference for  this topic is \cite{Cha}.
We recall, for the sake of completeness, some results about conditional distributions in Polish spaces.
Let $ Y $ and $ Z $ be two random variables defined on the same probability space $ (\Omega,\mathscr{F}, \mathbb{P}) $, with values, respectively, in the measurable spaces $ (E_1, \mathscr{E}_1) $ and $ (E_2, \mathscr{E}_2) $.  Let us denote by $ \mu_1$ the (marginal) laws of $ Y $,  by $ \mu_2$ the  marginal of $ Z $  and by $ \mu $ the joint distribution of $ (Y,Z) $ on $ (E, \mathscr{E})=(E_1\times E_2, \mathscr{E}_1\times \mathscr{E}_2). $ A family of probabilities $ (\mu_2(\cdot|y))_{y \in E_1} $ on $ (E_2, \mathscr{E}_2) $ is a regular version of the conditional law of $ Z $ given $ Y $ if
\begin{enumerate}
	\item For every $ B \in \mathscr{E}_2 $, the map ($(E_1, \mathscr{E}_1)\to (\R,\cl B(\R)$)), $ y \mapsto \mu_2(B|y) $ is $ \mathscr{E}_1$-measurable.
	\item For every $ B \in \mathscr{E}_2 $ and $ A \in \mathscr{E}_1 $, $ \mathbb{P}(Y \in A, Z \in B)=\int_A \mu_2(B|y)\mu_1(dy). $
\end{enumerate}
In this case we have
$$ \mu(dy,dz)=\mu_2(dz|y)\mu_1(dy). $$
In this section we will use the notation $ (E,\mathscr{B}) $ to indicate a Polish space (i.e. a separable, completely  metrizable space) with the Borel $ \sigma $-field, and we say that a sequence $ (x_n)_{n \in \mathbb{N}} \subset E $ converges to $ x \in E $, $ x_n \to x $, if $ d_E(x_n,x) \to 0 $, as $ n \to \infty $, where $ d_E $ denotes the metric on $ E $. Regular conditional probabilities do not always exist, but they exist in many cases. The following result, that immediately follows from Corollary 3.1.2 in \cite{Bor}, shows that in Polish spaces the regular version of the conditional probability is well defined.
\begin{proposition}\label{prop:conditional laws}
	 Let $ (E_1, \mathscr{B}_1) $ and $ (E_2, \mathscr{B}_2) $ be two Polish spaces endowed with their Borel $ \sigma $-fields, $ \mu $ be a probability measure on $ (E, \mathscr{B})= (E_1 \times E_2, \mathscr{B}_1 \times \mathscr{B}_2) $. Let $ \mu_i $,  $i=1,2 $, be the marginal probability measures on $ (E_i, \mathscr{B}_i)$. Then there exists $ \mu_1 $-almost sure a unique regular version of the conditional law of $ \mu_2 $ given $ \mu_1 $, i.e.
	 $$ \mu(dy,dz)=\mu_2(dz|y)\mu_1(dy) .$$
\end{proposition}
In what follows we always suppose random variables taking values in a Polish space. \\
	 Let $ (E_1, \mathscr{B}_1) $ and $ (E_2, \mathscr{B}_2) $ be two Polish spaces. We denote by $ (\mu_n)_{n \in \mathbb{N}} $ a sequence of probability measures on the product space $ (E_1 \times E_2, \mathscr{B}_1 \times \mathscr{B}_2) $ (the sequence of the \textit{joint distributions}), by $ (\mu_{in})_{n \in \mathbb{N}} $, for $ i=1,2 ,$ the sequence of the \textit{marginal distributions} on $ (E_i, \mathscr{B}_i) $ and by $ (\mu_{2n}(\cdot|x_1))_{n \in \mathbb{N}} $ the sequence of the  \textit{conditional distributions} on $ (E_2, \mathscr{B}_2) $ $ (x_1 \in E_1) $ given by Proposition \ref{prop:conditional laws}, i.e.
	 \begin{eqnarray}\label{eq:joint distribution}
	 \mu_n(B_1 \times B_2)=  \int_{B_1} \mu_{2n}(B_2|x_1)\mu_{1n}(dx_1)
	 \end{eqnarray}
	 for every $ B_1 \times B_2, $ with $ B_1 \in \mathscr{B}_1 $ and $ B_2 \in \mathscr{B}_2. $
	 \begin{definition}
	 	 Let $ (E_1, \mathscr{B}_1), \, (E_2, \mathscr{B}_2) $ be two Polish spaces and $ x_1 \in E_1. $ We say that the sequence of conditional laws $ (\mu_{2n}(\, \cdot \,|x_1))_{n \in \mathbb{N}} $ on $ (E_2, \mathscr{B}_2) $ satisfies the \textbf{LDP continuously} in $ x_1 $ with the rate function $ J(\, \cdot \,|x_1) $ and the speed $ \gamma_n $, or simply, the \textbf{LDP continuity condition} holds, if
	 	 \begin{enumerate}
	 	 	\item [\textit (a)] For each $ x_1 \in E_1 ,$ $ J(\, \cdot \,|x_1) $ is a good rate function on $ E_2 $.
	 	 	\item [\textit (b)]  For any sequence $ (x_{1n})_{n \in \mathbb{N}} $ in $ E_1 $ such that $ x_{1n} \to x_1 $, the sequence of measures $ (\mu_{2n}(\, \cdot \,|x_{1n}))_{n \in \mathbb{N}} $ satisfies a LDP on $ E_2 $ with the  rate function $ J(\, \cdot \,|x_1) $ and the speed $ \gamma_n $.
	 	 	\item [\textit (c)] $ J(\, \cdot \,|\, \cdot \,) $ is lower semicontinuous as a function of $ (x_1,x_2) \in E_1 \times E_2. $
	 	\end{enumerate}
	 \end{definition}

		\begin{theorem}\label{th:Chaganty}{\rm[Theorem 2.3 in \cite{Cha}]}
			Let $ (E_1, \mathscr{B}_1), \, (E_2, \mathscr{B}_2) $ be two Polish spaces. Let $ (\mu_{1n})_{n \in \mathbb{N}} $ be a sequence of probability measures on $ (E_1, \mathscr{B}_1)$. For $ x_1 \in E_1 $ let $ (\mu_{2n}(\cdot|x_1))_{n \in \mathbb{N}} $ be the  sequence of the conditional laws on $ (E_2, \mathscr{B}_2) $. Suppose that the following two conditions are satisfied:
			\begin{enumerate}
				\item [\textit (i)] $ (\mu_{1n})_{n \in \mathbb{N}} $ satisfies a LDP on $ E_1 $ with the good rate function $ I_1(\cdot) $ and the speed $ \gamma_n $.
				\item [\textit (ii)] for every $ x_1 \in E_1, $ the sequence $ (\mu_{2n}(\cdot|x_1))_{n \in \mathbb{N} }$ obeys the LDP continuity condition with the rate function $ J(\cdot|x_1) $ and the speed $ \gamma_n $.
			\end{enumerate}
		Then the sequence of joint distributions $ (\mu_n)_{n \in \mathbb{N}} $, given by (\ref{eq:joint distribution}), satisfies a WLDP on $ E=E_1 \times E_2 $ with the speed $ \gamma_n $ and the rate function  $$I(x_1,x_2)=I_1(x_1)+J(x_2|x_1) ,$$
			for $ x_1 \in E_1 $ and $ x_2 \in E_2 $. Furthermore the sequence of the marginal distributions $ (\mu_{2n})_{n \in \mathbb{N}} $ defined on $ (E_2, \mathscr{B}_2) $, satisfies a LDP with the speed $ \gamma_n $, and the rate function
			$$ I_2(x_2)=\inf_{x_1 \in E_1} I(x_1,x_2). $$
			Moreover, $ (\mu_n)_{n \in \mathbb{N}} $ satisfies a LDP if $ I(\,\cdot,\cdot\,) $ is a good rate function, and in this case, also $ I_2(\cdot) $ is a good rate function.
		\end{theorem}

We shall give a sufficient condition on the rate functions $ I_1(\cdot) $ and $ J(\cdot|\cdot) $ which guarantees that $ I(\,\cdot,\cdot\,) $ is a good rate function. See Lemma 2.6 in \cite{Cha}.
\begin{lemma}\label{lemma:good rate function}
In the same hypotheses of Theorem \ref{th:Chaganty},
	if the set
$$ \bigcup_{x_1 \in K_1}\{x_2: J(x_2|x_1)\leq L\} $$ is a compact subset of $ E_2 $ for any $ L \geq 0 $ and for any compact set $ K_1  \subset E_1, $
then $ I(\,\cdot,\cdot\,) $ is a good rate function (and therefore also $I_2(\cdot)$ is a good rate function).
	\end{lemma}

\section{Volterra type stochastic volatility models }\label{sect:SVM}
In the stochastic volatility models of  interest the dynamic of the asset price process $(S_t)_{t \in [0,T]}$ is modeled by the following equation
\begin{eqnarray}\label{eq:price-SDE}
\begin{cases}
dS_t = S_t \mu(\hat{B}_t)dt + S_t\sigma(\hat{B}_t)d(\bar{\rho}W_t+\rho B_t) \qquad 0\leq t \leq T,\\
S_0=s_0 >0,
\end{cases}
\end{eqnarray}
where  $s_0$ is the initial price, $ T > 0 $ is the time horizon,
$ \hat{B} $ is a non-degenerate continuous Volterra type  process as in (\ref{eq:integral representation}) for some kernel $ K $ which satisfies the conditions in Definition \ref{def:Volterra process},
the processes $ W $ and $ B $  are two independent standard Brownian motions, $ \rho \in (-1,1) $ is the correlation coefficient and $ \bar{\rho}=\sqrt{1-\rho^2}.$ Remark that $ \bar{\rho}W+\rho B $ is another standard Brownian motion which has correlation coefficient $ \rho $ with $ B $. If $ \rho \neq 0 $ the model is called a correlated stochastic volatility model, otherwise it is called an uncorrelated model.
The equation in (\ref{eq:price-SDE}) is considered on a filtered probability space $ (\Omega, \mathscr{F},(\mathscr{F}_t)_{0\leq t \leq T}, \mathbb{P});$ $(\mathscr{F}_t)_{0\leq t \leq T} $ is the filtration generated by $W$ and $B$, completed by the null sets, and  made right-continuous.
The filtration  $(\mathscr{F}_t)_{0\leq t \leq T} $ represents the information given by the two Brownian motions.
It is assumed in (\ref{eq:price-SDE}) that $ \mu: \mathbb{R} \to \mathbb{R} $ and $ \sigma: \mathbb{R} \to (0, +\infty) $ are continuous functions.
It follows from (\ref{eq:price-SDE}) that the process $ \sigma(\hat{B}) = (\sigma(\hat{B}_t))_{0\leq t \leq T}$ describes the stochastic evolution of volatility in the model and  the process
$\mu(\hat{B})=(\mu(\hat{B}_t))_{0\leq t \leq T} $ is an adapted return process.

Equation (\ref{eq:price-SDE}) has  a unique solution that can be represented as an exponential functional. The unique solution to the equation in (\ref{eq:price-SDE}) is the Dol\'{e}ans-Dade exponential
$$S_t=s_0 \exp \Big\{\int_0^t\mu(\hat{B}_s)\, ds-\frac 12 \int_0^t \sigma(\hat{B}_s)^2 \, ds +\bar{\rho} \int_0^t \sigma(\hat{B}_s)\, dW_s + \rho \int_0^t \sigma(\hat{B}_s)\, dB_s  \Big\}$$
for $  0\leq t \leq T $ (for further details, see Section IX-2 in \cite{Re-Yor}).
Therefore, the log-price process $ Z_t= \log S_t$, $ 0\leq t \leq T, $ with $ Z_0=x_0=\log s_0 $ is
\begin{eqnarray}\label{eq:X-explicit}
Z_t=x_0+ \int_0^t\mu(\hat{B}_s)\,ds-\frac 12\int_0^t\sigma(\hat{B}_s)^2 ds +  \bar{\rho}\int_0^t\sigma(\hat{B}_s) dW_s + \rho\int_0^t \sigma(\hat{B}_s) dB_s .
\end{eqnarray}

Let $ \ep_n: \mathbb{N} \to \mathbb{R}_+ $ be an infinitesimal  function.  For every $ n \in \mathbb{N} $, we will consider  the following scaled version of the stochastic differential equation in (\ref{eq:price-SDE})
$$\begin{cases}
dS_t^{n} =S_t^{n}\mu(\hat{B}_t^n)dt+ \ep_n  S_t^{n}\sigma(\hat{B}_t^n)d(\bar{\rho}W_t+\rho B_t) \qquad 0\leq t \leq T,\\
S_0^{n}=s_0.
\end{cases}$$
where, for every $ n \in \mathbb{N}, $
\begin{equation}\label{eq:Volterra B^n}\hat{B}^n_t=\ep_n  \hat B_t
,\quad  t\in[0,T].\end{equation}

In the next sections we will obtain a  sample path large deviation principle for the family of log-price processes $ ((Z_t^{n}-x_0)_{t \in [0,T]})_{n \in \mathbb{N}}.$
By contraction we will deduce a large deviation principle for the family $ (Z_T^{n}-x_0)_{n \in \mathbb{N}}$ obtaining the same result contained in \cite{Gu1}.

We will start by proving a  large deviation principle for the log-price in the uncorrelated stochastic volatility model and then we extend the results to the class of correlated models.
\section{LDP for the uncorrelated stochastic volatility model}\label{sect:uncorrelatedSVM}
We first consider, for $\bar\rho\neq 0$,   the model described by
\begin{eqnarray}\label{eq:price-SDE-uncorr}
\begin{cases}
dS_t = S_t \mu(\hat{B}_t)dt + \bar{\rho}S_t\sigma(\hat{B}_t)dW_t, \qquad 0\leq t \leq T,\\
S_0=s_0 >0,
\end{cases}
\end{eqnarray}
where the processes $ W $ and $ B $ driving, respectively, the stock price and the volatility equations  are two independent standard Brownian motions, so the model in (\ref{eq:price-SDE-uncorr}) is an uncorrelated stochastic volatility model.
 The corresponding scaled model is given by
\begin{eqnarray*}
\begin{cases}
dS_t^{n} = S_t^{n} \mu(\hat{B}_t^n)dt + \ep_nS_t^{n}\bar\rho\sigma(\hat{B}_t^n)dW_t, \qquad 0\leq t \leq T,\\
S_0^{n}=s_0 >0,
\end{cases}
\end{eqnarray*}
where, for every $ n \in \mathbb{N} $, $ \hat{B}^n $ is the Volterra process defined in  equation (\ref{eq:Volterra B^n}).
Moreover, the process $ X_t^{n}=\log S_t^{n} ,$ $ 0\leq t \leq T $, with $ X_0^{n}=x_0=\log s_0 $ is
\begin{equation}\label{eq:logprice-uncorr-scaled}
X_t^{n}=x_0+\int_0^t\Big(\mu(\hat{B}_s^n)-\frac 12\ep_n^{2}\sigma(\hat{B}_s^n)^2 \Big)\,ds +  \ep_n\bar\rho \int_0^t\sigma(\hat{B}_s^n) dW_s.
\end{equation}

We will prove that hypotheses of Chaganty's Theorem \ref{th:Chaganty} hold for the family of processes
$$  (\hat{B}^n,X^{n} -x_0)_{n \in \mathbb{N}}. $$

In order to guarantee the hypotheses of Chaganty's Theorem,
we will need to impose some  conditions on the coefficients.
First, let us recall, for future references, some well known facts on continuous functions.

\begin{remark}\label{rem:continuous}\rm
$(i)$
Suppose  $f:\R\to \R$ is a   continuous function and let
 $\varphi_n, \varphi\in C[0,T]$ be  functions
such that $ \varphi_n \overset{C[0,T]}{\underset{}{\longrightarrow}} \varphi, $ as $ n \to +\infty, $
then $ f\circ\varphi_n \overset{C[0,T]}{\underset{}{\longrightarrow}} f\circ\varphi, $ as $ n \to +\infty. $

$(ii)$
Suppose  $f:\R\to (0,+\infty)$ is a   continuous function  and let $(\varphi_n)_n\subset  C[0,T]$ be a sequence of equi-bounded functions, i.e., there exist $M>0$ such that for every $n\in\N$,
$\sup_{t\in[0,T]}|\varphi_n(t)| \leq M$, then there exist constants ${\underline f}_M,{\overline f}_M >0$ such that, for every $n\in\N$ and for every $t\in[0,T]$, $$0<{\underline f}_M\leq f(\varphi_n(t))\leq {\overline f}_M.$$

\end{remark}

\begin{assumption}\label{ass:hp-sigma-mu-I}
$ \sigma: \R \longrightarrow (0,+\infty) $ and $ \mu: \R \longrightarrow \R $  are continuous functions.
\end{assumption}

\begin{remark}\label{rem:hat-f}\rm
For $L>0$, denote by $D_{L}$ the level sets in the Cameron Martin space, i.e.
\begin{equation}\label{eq:ball-CM}
D_{L}=\{f\in H_0^1[0,T]: \lVert f\rVert^2_{H_0^1[0,T]}\leq L\}.\end{equation}
Then for $f\in D_{L}$, from the Cauchy-Schwarz inequality,
$$|\hat{f}(t)|=\bigg|\int_0^tK(t,s)\dot{f}(s)\,ds \bigg|\leq\lVert f\rVert_{H_0^1[0,T]}\bigg(\int_0^TK^2(t,s)\,ds \bigg)^{\frac12}. $$
	Therefore (thanks to conditions $(a)$ and $(b)$  in Definition \ref{def:Volterra process}) there exists a constant $ \hat L>0 $
such that  $$\sup_{f\in D_{L}}\sup_{t \in [0,T]}|\hat{f}(t)|\leq \hat L.$$
\end{remark}
Let $ \mu_{1n} $ denote the law induced by $ \hat{B}^n $ on
the Polish space $ (E_1, \mathscr{B}_1)=(C_0[0,T], \mathscr{B}(C_0[0,T])) $ and
 for  $ n \in \mathbb{N} $, let $ \mu_{2n} $ be the law induced by $ X^{n}-x_0$  on  $ (E_2, \mathscr{B}_2)=(C_0[0,T], \mathscr{B}(C_0[0,T]))$.
  Moreover, for (almost) every $ \varphi \in C_0[0,T] $,  $ n \in \mathbb{N},$
let $\mu_{2n}(\cdot|\varphi) $ be  the conditional  law of the process,
 $${X}^{n,\varphi}=X^{n}|(\hat{B}^n_t=\varphi(t) \quad 0 \leq t \leq T),$$
i.e. for $ \varphi \in C_0[0,T]$,  $ \mu_{2n}(\cdot|\varphi) $ is the law of the process
\begin{equation}\label{eq:logprice-uncorr-scaled-cond}
{X}^{n,\varphi}_t= x_0+ \int_0^t\Big(\mu(\varphi(s))-\frac 12 \ep^{2}_n \sigma(\varphi(s))^2\Big)\, ds + \ep_n\bar\rho   \int_0^t\sigma(\varphi(s))\, dW_s, \quad 0\leq t \leq T.\\
\end{equation}

Let's now check that the hypotheses of Theorem \ref{th:Chaganty} are fulfilled.
The sequence $ (\mu_{1n})_{n \in \mathbb{N}} $ satisfies a LDP on $ C_0[0,T] $ with  the  speed $ \ep_n^{-2}$ and the good rate function $ I_{\hat{B}}(\cdot) $ given by (\ref{eq:volterra rate function})  (condition $(i)$ of Theorem \ref{th:Chaganty}); therefore it is enough to show that the conditions $(a)$, $(b)$ and $(c)$ of LDP continuity condition are satisfied (condition $(ii)$ of Theorem \ref{th:Chaganty}).

\begin{proposition} \label{prop:LDP continuity condition Uncorrelated}
The sequence of the conditional laws $ (\mu_{2n}(\, \cdot \,|\varphi))_{n \in \mathbb{N}} $  satisfies, on $ C_0[0,T] $, the LDP continuity condition  with the rate function $ J( \cdot|\varphi) $ and and the inverse speed $ \ep_n^{2}$.
\end{proposition}
\begin{remark}\rm
Notice that, for every $ \varphi \in C_0[0,T] $ and $ n \in \mathbb{N} ,$ ${X}^{n,\varphi} $ is a Gaussian diffusion process. We will prove that
the sequence of conditional distributions obeys the conditions $(a)$ and $(b)$ of the LDP continuity condition by  using the generalized Freidlin-Wentzell's Theorem \ref{th:general W-F}. The same result can be obtained by using theory of Gaussian processes. For major details
see, for example, Section 5 in \cite{Pac-Pig}.
\end{remark}

\proofof{Proposition \ref{prop:LDP continuity condition Uncorrelated}}

$(a)$  For  $ \varphi \in C_0[0,T] $ we check that
$ (\mu_{2n}(\cdot|\varphi))_{n \in \mathbb{N}} $ obeys a  LDP   on $ C_0[0,T] $ with the good rate function $J(\cdot|\varphi)$.
With the same notation of  Theorem \ref{th:general W-F}, we have
\begin{itemize}
	\item $ b_n(t)=\mu(\varphi(t))-\frac 12 \ep^{2}_n \sigma(\varphi(t))^2$, then $b_n(t) \to \mu(\varphi(t)) $, as $ n \to +\infty, $ uniformly for $ t \in [0,T] $;
	\item $c_n(t)=\sigma(\varphi(t)) $, not depending on $ n. $
\end{itemize}
Then the family $ (\mu_{2n}(\cdot|\varphi))_{n\in \mathbb{N}} $ satisfies a  LDP on $ C_0[0,T] $ with the inverse speed $ \ep^{2}_n $ and the good rate function
$$J(x|\varphi) = \inf \Big\{\frac 12  \int_0^T \dot{y}(t)^2 \, dt:\, \int_0^t\mu(\varphi(s))\,ds+\bar\rho\int_0^t\sigma(\varphi(s))\dot{y}(s)\, ds=x(t), \,  y \in H^1_0[0,T]\Big\}$$
with the usual  understanding $J(x|\varphi)= +\infty  $ if the set is empty.
If $ y \in H_0^1[0,T] $ then
$$ \dot{x}(t)=\mu(\varphi(t))+\bar\rho\,\sigma(\varphi(t))\dot{y}(t) \,\,\mbox{ a.e., with } x(0)=0.$$
Thanks to Remark \ref{rem:continuous} $(ii)$, $\sigma\circ \varphi>0$ and the rate above simplifies to
\begin{eqnarray}\label{eq:cond-rate-function} J(x|\varphi)=\begin{cases}
\displaystyle\frac 12 \int_0^T \bigg(\frac{\dot{x}(t)-\mu(\varphi(t))}{\bar\rho\,\sigma(\varphi(t))}\bigg)^2  dt & x \in H^1_0[0,T]\\

+\infty &\mbox{ otherwise}.
\end{cases}
\end{eqnarray}

\medskip

$(b)$  Let $ (\varphi_n)_{n \in \mathbb{N}} \subset C_0[0,T] $ and $ \varphi \in C_0[0,T] $ be functions such that $ \varphi_n \overset{C_0[0,T]}{\underset{}{\longrightarrow}} \varphi, $ as $ n \to +\infty. $ We  check that the sequence  $(\mu_{2n}(\cdot|\varphi_n))_{n \in \mathbb{N}}$ obeys a LDP on $ C_0[0,T], $ with the (same) rate function $ J(\cdot|\varphi). $
For every $ n \in \mathbb{N} $, denote by ${X}^{n,\varphi_n} $ the process
$$
X^{n,\varphi_n}_t=x_0+\int_0^t\mu(\varphi_n(s))\,ds -\frac 12 \ep^{2}_n\int_0^t\sigma(\varphi_n(s))^2\, ds + \ep_n\bar \rho\int_0^t \sigma(\varphi_n(s))\, dW_s, \quad 0\leq t \leq T .$$
With the same notation of  Theorem \ref{th:general W-F}, thanks to Remark \ref{rem:continuous} $(i)$,  we have
\begin{itemize}
	\item $ b_n(t)=\mu(\varphi_n(t))-\frac 12 \ep^{2}_n \sigma(\varphi_n(t))^2$, then $b_n(t) \to \mu(\varphi(t)) $, as $ n \to +\infty, $ uniformly for $ t \in [0,T] $;
	\item $ c_n(t)=\bar \rho\,\sigma(\varphi_n(t)) $, then $c_n(t) \to \bar \rho\,\sigma(\varphi(t))$, as $ n \to +\infty, $ uniformly for $ t \in [0,T]. $
\end{itemize}
Therefore
 $ (\mu_{2n}(\cdot|\varphi_n))_{n \in \mathbb{N}} $ obeys a LDP with the inverse speed $ \ep^{2}_n $ and the good rate function $ J(\cdot|\varphi).$

\medskip

$ (c) $ We  check that $ J(\cdot|\cdot) $ is lower semicontinuous as a function of the couple $ (\varphi,x) \in C_0[0,T]^2. $\\
Suppose that
$$  (\varphi_n,x_n) \overset{C_0[0,T]^2}{\underset{n \to +\infty}{\longrightarrow}} (\varphi,x).  $$
If $\liminf_{n \to +\infty} J(x_n|\varphi_n)=\lim_{n \to +\infty} J(x_n|\varphi_n)=+\infty,$ there is nothing to prove. Therefore we can suppose that $(x_n)_{n\in \N}\subset H^1_0[0,T]$ and then
\begin{eqnarray*}
 J(x_n|\varphi_n)&=& \frac 12 \int_0^T\bigg(\frac{\dot{x}_n(t)-\mu(\varphi_n(t))}{\bar\rho\,\sigma(\varphi_n(t))}\bigg)^2 \, dt=\\
 &=&\frac 12 \int_0^T\bigg(\frac{\dot{x}_n(t)-\mu(\varphi_n(t))}{\bar\rho\,\sigma(\varphi(t))}\bigg)^2\cdot\bigg(\frac{\sigma(\varphi(t))}{\sigma(\varphi_n(t))}\bigg)^2 \, dt\geq\\
 &\geq& \inf_{t \in [0,T]} \bigg( \frac{\sigma(\varphi(t))}{\sigma(\varphi_n(t))}\bigg)^2\cdot\frac 12  \int_0^T\bigg(\frac{\dot{x}_n(t)-\mu(\varphi_n(t))}{\bar\rho\,\sigma(\varphi(t))}\bigg)^2 \, dt\\
 &=&\inf_{t \in [0,T]} \bigg( \frac{\sigma(\varphi(t))}{\sigma(\varphi_n(t))}\bigg)^2 J\Big(x_n+\int_0^{\cdot}(\mu(\varphi(s))- \mu(\varphi_n(s))ds\Big|\varphi\Big).
\end{eqnarray*}
Now $x_n+\int_0^{\cdot}(\mu(\varphi(s))- \mu(\varphi_n(s))ds\to x$ in $C_0[0,T]$, as $n\to +\infty$.
Therefore  from the semicontinuity of $J(\cdot|\varphi)$ (it is a rate function) and Remark  \ref{rem:continuous}
the claim follows.
\cvd

The hypotheses of Theorem \ref{th:Chaganty} are fulfilled, so we have the following result.
\begin{proposition} The family
	$$ (\hat{B}^n,X^{n}-x_0)_{n \in \mathbb{N}} $$
satisfies a  WLDP
with the speed $ \ep_n^{-2} $ and the rate function
	$$ I(\varphi,x)=I_{\hat{B}}(\varphi)+J(x|\varphi).  $$
Furthermore $ (X^{n}-x_0)_{n \in \mathbb{N}} $ satisfies a LDP  with the speed function $ {\ep^{-2}_n} $ and the rate function
\begin{eqnarray*}
I_X(x)=\begin{cases} \displaystyle\inf_{f\in H_0^1[0,T]}\left[\frac12 \lVert f\rVert_{H_0^1[0,T]}^2 + \frac12  \int_0^T\Bigg(\frac{\dot{x}(t)-\mu(\hat{f}(t))}{\sigma(\hat{f}(t))} \Bigg)^2 \,dt \right] & x \in H_0^1[0,T]\\
		\displaystyle\phantom{\inf} +\infty & x \notin H_0^1[0,T],
		\end{cases}
		\end{eqnarray*}
	where  $\hat f $ is defined in (\ref{eq:hat-f}).

\end{proposition}
\proof
	Thanks to Remark \ref{rem:LDP Volterra} and Proposition \ref{prop:LDP continuity condition Uncorrelated},  the family $ (\hat{B}^n,X^{n}-x_0)_{n \in \mathbb{N}} $ satisfies the hypotheses of Chaganty's Theorem \ref{th:Chaganty} and therefore  satisfies a WLDP with the speed $ \ep_n^{-2} $ and the rate function
$$ I(\varphi,x)=I_{\hat{B}}(\varphi)+J(x|\varphi),  $$
for $x\in C_0[0,T]$ and $\varphi\in C_0[0,T].$
Furthermore $ (X^{n}-x_0)_{n \in \mathbb{N}} $ satisfies a LDP on $ C_0[0,T] $ with the speed function $ {\ep^{-2}_n} $ and the rate function
$$I_X(x)= \inf_{\varphi\in C_0[0,T]}I(\varphi,x).  $$
From the expressions of the rate functions $ I_{\hat{B}}(\cdot) $ in (\ref{eq:volterra rate function}) and $ J(\cdot|\cdot) $ in (\ref{eq:cond-rate-function}),  the claim follows. \cvd

Now we show that  $ I_X(\cdot) $ is a good rate function and this follows from Lemma \ref{lemma:good rate function}.
\begin{lemma}
	Let $ J: C_0[0,T]\times C_0[0,T] \longrightarrow [0, +\infty] $ be defined in (\ref{eq:cond-rate-function}),
	Then, for any $ L \geq 0 $ and for any compact set $ K_1  \subset C_0[0,T]$,
	$$  \bigcup_{\varphi \in K_1}\{x \in C_0[0,T]: J(x|\varphi)\leq L\} $$
	is a compact subset of $ C_0[0,T]$, 	
therefore  $ I_X(\cdot) $ is a good rate function.
\end{lemma}
\proof

	Let  $ K_1 $ be a compact set of $ C_0[0,T] $. For $L\geq 0$  let us define
	$$ A_\varphi^L= \{x \in C_0[0,T]: J(x|\varphi)\leq L\} = \{x \in H_0^1[0,T]: J(x|\varphi)\leq L\}.$$
For every $ \varphi \in K_1 $, $ A_\varphi^L $ is a compact subset of $ C_0[0,T]$ (since $ J(\cdot|\varphi) $ is a good rate function).
If $ (x_n)_{n \in \mathbb{N}} \subset  \bigcup_{\varphi \in K_1} A_\varphi^L $, then, for every $ n \in \mathbb{N}, $ there exists $ \varphi_n \in K_1 $ such that $ x_n \in A_{\varphi_n}^L $ (i.e. $ J(x_n|\varphi_n)\leq L $).
	The sequence $ (\varphi_n)_{n \in \mathbb{N}}$ is contained in  $ K_1$, therefore we can suppose   that $ {\varphi}_n \overset{C_0[0,T]}{\underset{}{\longrightarrow}} \tilde\varphi \in K_1$, as $ n\to +\infty. $
Straightforward computations show that exists $M>0$ such that, for every $n\in \N$,
$$ J(x_n|\tilde\varphi)\leq M.$$
	Therefore  for every $ n \in \N$, $ x_n \in A_{\tilde\varphi}^{M}, $ which is a compact set. Then, up to a subsequence, we can suppose that $ {x_{n}} \to x \in A_{\tilde\varphi}^{M}, $ as $ n \to +\infty. $
	Since $J(\cdot|\cdot)$ is semicontinuous, then
$$J(x|\tilde\varphi)\leq \liminf_{n\to +\infty} J(x_n|\varphi_n)\leq L,$$
i.e.  $x\in  A_{\tilde\varphi}^L\subset\bigcup_{\varphi \in K_1} A_\varphi^L $ and therefore $\bigcup_{\varphi \in K_1} A_\varphi^L$ is compact.
We are in the hypotheses of Lemma \ref{lemma:good rate function}, then $ I(\cdot,\cdot) $ is a good rate function, and also $ I_X(\cdot) $ is a good rate function.\cvd

We summarize the sample path large deviation principle for the process $ (X^{n}-x_0)_{n \in \mathbb{N}} $  in the following theorem.
\begin{theorem}\label{th:main-uncorr}
Under  Assumptions \ref{ass:hp-sigma-mu-I}
	a  large deviation principle with the speed $ \ep^{-2}_n $ and the good rate function
$$ I_X(x)=\begin{cases}\displaystyle \inf_{f\in H_0^1[0,T]}\left[\frac12 \lVert f\rVert_{H_0^1[0,T]}^2 + \frac12  \int_0^T\Bigg(\frac{\dot{x}(t)-\mu(\hat{f}(t))}{\sigma(\hat{f}(t))} \Bigg)^2 \,dt \right]& x \in H_0^1[0,T] \\
		\displaystyle +\infty & x \notin H_0^1[0,T]
		\end{cases}
		$$
	holds for the family $ (X^{n}-x_0)_{n \in \mathbb{N}} $, where for every $ n \in \mathbb{N} $, $ (X_t^{n})_{t \in [0,T]} $ is defined by (\ref{eq:logprice-uncorr-scaled}).
\end{theorem}

\section{LDP for the correlated Stochastic Volatility Model}\label{sect:correlatedSVM}

We now consider a correlated Volterra type stochastic volatility model. The asset price process $ (S_t)_{t \in [0,T]} $ is modeled by the following stochastic differential equation
\begin{eqnarray*}
\begin{cases}
dS_t = S_t \mu(\hat{B}_t)dt + S_t\sigma(\hat{B}_t)d(\bar{\rho}W_t+\rho B_t) \qquad 0\leq t \leq T,\\
S_0=s_0 >0,
\end{cases}
\end{eqnarray*}
where $ \rho \in (-1,1) $ (for $ \rho= 0 $ we have the uncorrelated model).
Like before, let $ Z_t=\log S_t $ $, 0\leq t \leq T, $ be the log-price process defined by (\ref{eq:X-explicit}). We are going to consider the following  process
\begin{equation}\label{eq:logprice-corr-scaled}
Z_t^{n}=x_0+\int_0^t\Big(\mu(\hat{B}^n_s)-\frac 12\ep_n^{2}\sigma( \hat{B}^n_s)^2\Big) ds +  \ep_n \bar{\rho}\int_0^t\sigma( \hat{B}^n_s) dW_s
+ \ep_n \rho\int_0^t\sigma( \hat{B}^n_s) d B_s,
\end{equation}
where $ Z_0^{n}=x_0=\log s_0. $ In this section we want to prove a sample path large deviation principle for the family  $ ((Z^{n}_t-x_0)_{t\in [0,T]})_{n \in \mathbb{N}}.$
 Notice that
$$Z_t^{n}=X_t^n
+ \ep_n \rho\int_0^t\sigma(\ep_n \hat{B}_s) d B_s,$$
where $ (X^n_t)_{t\in[0,T]}$ is defined in (\ref{eq:logprice-uncorr-scaled}).
The study of the correlated model is more complicated than the previous one. In fact, in this case, we should also study the behavior of the process $ (V^n_t)_{t\in[0,T]}$ where
\begin{equation} \label{eq:V_n}V^n_t=\ep_n  \rho\int_0^t\sigma(\ep_n \hat{B}_s)dB_s, \quad 0\leq t\leq T.\end{equation}
Notice that this process depends on the couple $(\ep_n B,\ep_n \hat{B})$, but we can't directly apply  Chaganty's Theorem to the family
$$ ((\ep_n B,\ep_n \hat{B}), Z^{n}-x_0)_{n \in \mathbb{N}} $$
since
$ V^n $ cannot be written as a continuous function of $ (\ep_n B,\ep_n \hat{B}) $ and so the LDP continuity condition is not fulfilled. To overcome this problem, we  introduce a new family of processes $ (Z^{n,m})_{n \in \mathbb{N}} $, where for every $ m \geq 1 $, $V^n$
is replaced by a suitable continuous function  of $ (\ep_n B,\ep_n \hat{B}) $.
Thanks to the results obtained in the previous section, we  prove that the hypotheses of Chaganty's Theorem are fulfilled for the family $ ((\ep_n B,\ep_n \hat{B}), Z^{n,m}-x_0)_{n \in \mathbb{N}} $. Then, for every $ m \geq 1 $, $ (Z^{n,m}-x_0)_{n \in \mathbb{N}} $ satisfies a LDP with a certain good rate function $ I^m$ (Section 6.1). Then, proving that the family $ ((Z^{n,m})_{n \in \mathbb{N}})_{m\in \N} $ is an exponentially good approximation of $ (Z^{n})_{n \in \mathbb{N}} $, we  obtain a large deviation principle for the family $ (Z^{n}-x_0)_{n \in \mathbb{N}} $ with the good rate function obtained in terms of the $ I^m $'s (Section 6.2). Finally (in Section 6.3) we give an explicit expression for the rate function (not in terms of the $ I^m $'s). In Section 6.4 we give an application of the previous results.

In this section we  need some more hypotheses on coefficients $\mu$ and $\sigma$.
\begin{definition}
A modulus of continuity is an increasing function  $\omega:[0,+\infty)\to[0,+\infty)$
such that $\omega(0) = 0 $ and $\lim_{x\to 0 }\omega (x)=0$.
 A function $f$  defined on $\R$  is
called locally $\omega$-continuous, if for every $\delta>0$ there exists a constant $L(\delta)>0$  such that for all
$x,y\in[-\delta, \delta]$, the following inequality holds: $|f(x)-f(y)|\leq L(\delta)\omega(|x-y|)$.
\end{definition}
\begin{remark}\rm
For instance, if $\omega(x)=x^{\alpha}$, $\alpha\in(0,1)$,   the function $f$ is locally $\alpha $-H\"{o}lder continuous.
If $\omega(x)=x$,    the function $f$ is locally Lipschitz continuous.
\end{remark}

Consider the following assumptions.

\begin{assumption}\label{ass:hp-sigma-II}\rm

	$ \sigma: \R \longrightarrow (0,+\infty) $ is a locally $\omega$-continuous function.

\end{assumption}

\begin{assumption}\label{ass:hp-sigma-mu-III}\rm
		 There exist constants $\alpha, M_1,M_2>0, $ such that
		$$ \sigma(x)+ |\mu(x)|\leq M_1+ M_2 \,|x|^\alpha, \quad  x\in \R.$$
			
\end{assumption}

\subsection{LDP for the approximating families}\label{sect:appr-fam}

In this section we suppose that  Assumptions \ref{ass:hp-sigma-mu-I}  are fulfilled.

\medskip

For every $ m\geq1 $, let us define the  functions $ \Psi_m:C_0[0,T]^2\to C_0[0,T] $,
\begin{equation}\label{eq:psi-m}
 \Psi_m(f,g)(t)= \sum_{k=0}^{\left\lfloor \frac{mt}{T}\right\rfloor-1}\sigma\Big( g\Big(\frac km T\Big) \Big)\Big[f \Big(\frac{k+1}m T \Big)- f\Big( \frac km T\Big) \Big]+
\sigma\Big(g\Big(\Big\lfloor \frac{mt}{T}\Big\rfloor\frac Tm\Big) \Big)\Big[f(t)- f\Big( \Big\lfloor\frac{mt}{T}\Big\rfloor\frac Tm\Big) \Big],
\end{equation}
$t\in[0,T]$.
We note that, for every $ m\geq 1 $, $ \Psi_m $ is a continuous function on  $ C_0[0,T]^2 $ (where we are using the sup-norm topology for both arguments of $ \Psi_m $).
\begin{remark}\label{rem:psi-m}
	\rm
		If $ (f,g)\in \cl H_{(B,\hat B)} $, i.e. $ f \in H_0^1[0,T] $ and $ g=\hat{f} $ where $\hat{f} $ is defined in (\ref{eq:hat-f}),
 the function $ \Psi_m $ can be written
		$$ \Psi_m(f,\hat{f})(t)=\int_0^t \sigma\Big(\hat{f}\Big({\Big\lfloor \frac{ms}T\Big\rfloor}\frac Tm\Big)\Big)\dot{f}(s)ds $$
		for $t \in [0,T].  $
		
\end{remark}

In order to establish a LDP for the family $((Z^{n}_t-x_0)_{t\in [0,T]})_{n \in \mathbb{N}},  $ we introduce a new family of processes $ ((Z_t^{n,m})_{t \in [0,T]})_{n \in \mathbb{N}} $, $ m\geq 1 $,
defined by
\begin{equation}\label{eq:logprice-scaled-corr-psi-m}
Z_t^{n,m}= x_0 + \int_0^t \Big(\mu(\ep_n\hat{B}_s)ds -\frac12 \ep_n^{2} \sigma(\ep_n \hat{B}_s)^2 \Big)ds + \ep_n\bar{\rho}\int_0^t\sigma(\ep_n\hat{B}_s)dW_s
+\rho \Psi_m(\ep_n B,\ep_n\hat{B})(t).
\end{equation}
We  prove a large deviation principle for the family $ ((Z_t^{n,m}-x_0)_{t \in [0,T]})_{n \in \mathbb{N}}$ (for  $m\geq 1$), as $ n \to +\infty. $
For this purpose we check that hypotheses of  Theorem \ref{th:Chaganty} hold for the family of processes
$$ ((\ep_nB,\ep_n\hat{B}), Z^{n,m}-x_0)_{n \in \mathbb{N}}. $$
From Theorem \ref{th:LDP couple} we already know that the couple $ ((\ep_nB,\ep_n\hat{B}))_{n \in \mathbb{N}} $ satisfies a large deviation principle on $ C_0[0,T]^2 $ with the inverse speed $ \ep_n^{2} $ and
 the good rate function $ I_{(B,\hat{B})}(\cdot,\cdot) $ given by (\ref{eq:couple rate function}). For fixed $ m\geq1 $ and $ (f,g)\in C_0[0,T]^2  $ our next goal is to prove that the family of conditional processes
$$ {Z}^{n,m,(f,g)} = Z^{n,m}|(\ep_nB_t=f(t), \ep_n\hat{B}_t=g(t) \quad 0\leq t \leq T)  $$
satisfies a large deviation principle.
For every $ (f,g) \in  C_0[0,T]^2  $ and $ t \in [0,T] $ we have
\begin{equation}\label{eq:logprice-scaled-corr-psi-m-cond}
 {Z}_t^{n,m,(f,g)}=x_0+\int_0^t\Big(\mu(g(s))ds-\frac12\ep_n^{2}\sigma(g(s))^2\Big)ds + \ep_n\bar{\rho}\int_0^t\sigma(g(s))dW_s+
  \rho\Psi_m(f,g)(t),
\end{equation}
i.e.
$$  {Z}_t^{n,m,(f,g)}=  X_t^{n, g} +\rho\Psi_m(f,g)(t),$$
where  $(  X_t^{n, g})_{t\in[0,T]}$ is defined  in (\ref{eq:logprice-uncorr-scaled-cond}) and the equalities are to be intended in law.

\begin{proposition}\label{prop:LDP-logprice-scaled-corr-psi-m-cond}
If $ (f,g) \in C_0[0,T]^2 $,
	 then for every  $ m \geq 1 $, $(( {Z}_t^{n,m,(f,g)}- x_0)_{t \in [0,T]})_{n \in \mathbb{N}} $ satisfies a large deviation principle on $ C_0[0,T] $ with the speed $ \ep_n^{-2} $ and the good rate function

\begin{equation}\label{eq:rate-function-logprice-scaled-corr-psi-m-cond} {\cal J}^m(x|(f,g))= J(x - \rho\Psi_m(f,g)|g), \end{equation}
	where $ J(\cdot|g) $ is given by (\ref{eq:cond-rate-function}).
\end{proposition}
\proof
Combining Proposition \ref{prop:LDP continuity condition Uncorrelated}  and the contraction principle the proof is complete. Note that
$ {\cal J}^m(x|(f,g)) $ is finite  when
 $ x-\rho\Psi_m(f,g)\in H_0^1[0,T]. $\cvd

\begin{remark}\label{rem:rate-function-logprice-scaled-corr-psi-m-cond}
	\rm If $ (f,g)\in \cl H_{(B,\hat B)}$ then, we have already seen in Remark \ref{rem:psi-m} that the function $ \Psi_m(f,g) $ can be written as
$$ \Psi_m(f,\hat{f})(t)=\int_0^t \sigma\left(\hat{f}\left({\Big\lfloor \frac{ms}T\Big\rfloor}\frac Tm\right)\right)\dot{f}(s)ds $$
		for $t \in [0,T].  $ Clearly $ \Psi_m(f,\hat{f}) $ is differentiable with a square integrable derivative, i.e. $ \Psi_m(f,\hat{f})\in H_0^1[0,T] .$
Therefore, the rate function $ {\cal J}^m(\cdot|(f,\hat{f})) $ is
		\begin{eqnarray}\label{eq:Jm explicit} {\cal J}^m(x|(f,\hat{f}))=\begin{cases}
		\displaystyle \frac12\int_0^T\Big( \frac{\dot{x}(t)- \mu(\hat{f}(t))-\rho\dot{\Psi}_m(f,\hat{f})(t)}{\bar{\rho}\sigma(\hat{f}(t))} \Big)^2dt & x \in H_0^1[0,T]\\
		+\infty & x \notin H_0^1[0,T].
		\end{cases}
		\end{eqnarray}
	\end{remark}
\begin{proposition}\label{prop:II LDP continuity condition correlated }
	Let $((f_n,g_n))_{n \in \mathbb{N}} \subset C_0[0,T]^2, $ $ (f,g) \in C_0[0,T]^2 $ be functions such that
	$ (f_n,g_n)\overset{C_0[0,T]^2}{\underset{}{\longrightarrow}} (f,g)$.
	 Then, for every $ m\geq 1$, the family of processes $ (( {Z}_t^{n,m,(f_n,g_n)}-x_0)_{t \in [0,T]})_{n \in \mathbb{N}}$, where
	$$  {Z}_t^{n,m,(f_n,g_n)}=\int_0^t\Big(\mu(g_n(s))ds-\frac12\ep_n^{2}\sigma(g_n(s))^2\Big)ds + \ep_n\bar{\rho}\int_0^t\sigma(g_n(s))dW_s+\rho\Psi_m(f_n,g_n)(t)  $$
	satisfies a large deviation principle on $ C_0[0,T] $ with the speed $ \ep_n^{-2} $ and the good rate function ${\cal J}^m(\cdot|(f,g))$ defined in (\ref{eq:rate-function-logprice-scaled-corr-psi-m-cond}).
\end{proposition}
\begin{proof}
	Since $ g_n \overset{C_0[0,T]}{\underset{}{\longrightarrow}} g, $ as $ n \to +\infty$, from Proposition  \ref{prop:LDP continuity condition Uncorrelated} (condition $(b)$ of LDP continuity condition),  we already know that the family $ (( {X}_t^{n,g_n})_{t \in [0,T]})_{n \in \mathbb{N}} $ satisfies a LDP with the speed $ \ep_n^{-2} $ and the good rate function
	$$ J(x|g)=\begin{cases}
	\displaystyle \frac12\int_0^T\Big( \frac{\dot{x}(t)- \mu(g(t))}{\bar{\rho}\sigma(g(t))}\Big)^2\,dt & x \in H_0^1[0,T]\\
	+\infty &x \notin H_0^1[0,T]
	\end{cases} $$
	Combining this with the contraction principle, we have that
	the family  $$(( {X}_t^{n,g_n}+\rho\Psi_m(f,g)(t))_{t \in [0,T]})_{n \in \mathbb{N}} $$ satisfies a large deviation principle with the speed $ \ep_n^{-2} $ and the good rate function $  {\cal J}^m(x|(f,g)) $ for  $ x \in C_0[0,T]. $
Furthermore, for every $ m\geq1 $
$$ \Psi_m(f_n,g_n)\overset{}{\underset{n \to +\infty}{\longrightarrow}} \Psi_m(f,g) $$
in $C_0[0,T]  $, since $ \Psi_m $ is a continuous function. Therefore the  families
 $ (( {X}_t^{n,g_n}+\rho\Psi_m(f_n,g_n)(t))_{t \in [0,T]})_{n \in \mathbb{N}}$ and
  $  (( {X}_t^{n,g_n}+\rho\Psi_m(f,g)(t)_{t \in [0,T]})_{n \in \mathbb{N}}  $ are  exponentially equivalent (see Remark \ref{exp-equiv})
and  the proof is complete. \cvd
\end{proof}

We now want to prove the lower semicontinuity of $ {\cal J}^m(\cdot|(\cdot,\cdot)). $
\begin{proposition}\label{prop:III LDP continuity condition correlated }
	If $ ((f_n,g_n),x_n)   \overset{}{\underset{n \to +\infty}{\longrightarrow}}    ((f,g),x) $, in $ C_0[0,T]^2\times C_0[0,T] $, then for every $ m\geq 1 $
	$$ \liminf_{n \to +\infty}{\cal J}^m(x_n|(f_n,g_n))\geq {\cal J}^m(x|(f,g)). $$
\end{proposition}
\begin{proof}
	From Proposition \ref{prop:LDP-logprice-scaled-corr-psi-m-cond} we have
	$${\cal J}^m(x_n|(f_n,g_n))= J(x_n- \rho \Psi_m(f_n,g_n)|g_n). $$
Recall that for every $ m\geq1  $, $ \Psi_m $ is continuous on $ C_0[0,T]^2 $. Therefore, if $ ((f_n,g_n),x_n) \to ((f,g),x) $, as $ n \to +\infty, $ in $ C_0[0,T]^2\times C_0[0,T]$, then
	$$ x_n - \rho \Psi_m(f_n,g_n) \overset{}{\underset{n \to +\infty}{\longrightarrow}} x-\rho \Psi_m(f,g) $$
	in $ C_0[0,T] $. Then, from the lower semicontinuity of $ J(\cdot|\cdot) $ (see  Proposition  \ref{prop:LDP continuity condition Uncorrelated} (condition $(c)$ of LDP continuity condition))
	$$
 \liminf_{n \to +\infty}{\cal J}^m(x_n|(f_n,g_n))=\liminf_{n \to +\infty}J(x_n- \rho \Psi_m(f_n,g_n)|g_n),
 \geq J(x- \rho \Psi_m(f,g)|g)={\cal J}^m(x|(f,g)),
	$$
	which concludes the proof.\cvd
\end{proof}

\begin{proposition}
	For $ m\geq 1 $, the family
	$ ((\ep_nB,\ep_n\hat{B}),Z^{n,m}-x_0 )_{n \in \mathbb{N}} ,$
	where for every $ n \in \mathbb{N}, $ $Z^{n,m}  $	is the process defined by (\ref{eq:logprice-scaled-corr-psi-m})
	 satisfies a WLDP with the speed $ \ep_n^{-2} $ and the rate function
	\begin{equation}\label{eq:H^mjoint} {\cal H}^m((f,g),x)=I_{(B,\hat{B})}(f,g)+{\cal J}^m(x|(f,g))\end{equation}
	for $ x \in C_0[0,T] $ and $ (f,g)\in C_0[0,T]^2 ,$ and $ (Z^{n,m}-x_0 )_{n \in \mathbb{N}} $ satisfies a LDP with the speed $ \ep_n^{-2} $ and the rate function
	\begin{eqnarray}\label{eq:rate-function-logprice-scaled-corr-psi-m}
	I_{Z}^m(x)=\begin{cases}\displaystyle
	\inf_{f \in H_0^1[0,T]}    \left\{  \frac12   \int_0^T    \dot{f}(u)^2\, du  +  \frac12  \int_0^T \Big(\frac{\dot{x}(t) - \mu(\hat{f}(t)) - \rho\dot{\Psi}_m(f,\hat{f})(t)}{\bar{\rho}\sigma(\hat{f}(t))} \Big)^2   dt\right\}
&  x \in H_0^1[0,T]\\
+\infty &x \notin H_0^1[0,T].\end{cases}
	\end{eqnarray}
	
\end{proposition}
\proof
	Thanks to Theorem \ref{th:LDP couple} and Propositions \ref{prop:LDP-logprice-scaled-corr-psi-m-cond}, \ref{prop:II LDP continuity condition correlated } and \ref{prop:III LDP continuity condition correlated }, the family $ ((\ep_nB,\ep_n\hat{B}),Z^{n,m}-x_0 )_{n \in \mathbb{N}} $ satisfies the hypotheses of  Theorem \ref{th:Chaganty}. Therefore $ (Z^{n,m}-x_0 )_{n \in \mathbb{N}} $ satisfies a LDP with the speed $ \ep_n^{-2} $ and the rate function
\begin{equation}\label{eq:I^mZ} I_{Z}^m(x)=\inf_{(f,g)\in C_0[0,T]^2} \left\{I_{(B,\hat{B})}(f,g)+ {\cal J}^m(x|(f,g))\right\}\end{equation}
	for  $x \in C_0[0,T]$. From (\ref{eq:couple rate function})  and Remark \ref{rem:rate-function-logprice-scaled-corr-psi-m-cond} the claim follows.\cvd

We conclude this section showing that, for every $m\geq 1$, $I^m$ is a \it good \rm rate function.
 \begin{proposition}\label{prop:Im good}
	For every $ m \geq 1 ,$ the rate function $ I_Z^m(\cdot) $ defined in (\ref{eq:rate-function-logprice-scaled-corr-psi-m}) is a good rate function.
\end{proposition}
For the proof of this proposition, we need the following lemma.
\begin{lemma}\label{lemma:good rate}
	Let $ {\cal J}^m: C_0[0,T]^2 \times C_0[0,T] \longrightarrow [0, +\infty] $ be the rate function defined in (\ref{eq:Jm explicit}).
	Then, the set
	$$ \displaystyle\bigcup_{(f,g) \in K_1}\{x \in C_0[0,T]: {\cal J}^m(x|(f,g))\leq L\} $$
	is a compact subset of $ C_0[0,T] $ for any $ L \geq 0 $ and for any (compact) level set $K_1 $ of the  (good) rate function $I_{(B,\hat{B})}(\cdot,\cdot)  $ defined in (\ref{eq:couple rate function}).	
\end{lemma}
\proof
	Let $ K_1 $ be a level set of $ I_{(B,\hat{B})}(\cdot,\cdot) $.  For $L\geq 0$ and $ (f,g) \in K_1 $ define
	\begin{equation}\label{eq:AL} A_{(f,g)}^L= \{x \in C_0[0,T]: {\cal J}^m(x|(f,g))\leq L\}.
	\end{equation}
For every $ (f,g) \in K_1 $, $ A_{(f,g)}^L $ is a compact subset of $ C_0[0,T], $ since $ {\cal J}^m(\cdot|(f,g)) $ is a good rate function.
From the expression of the rate function $ I_{(B,\hat{B})}(\cdot,\cdot) $, we can deduce that $ K_1\ \subset  \mathscr{H}_{(B,\hat B)} $, where $ \mathscr{H}_{(B,\hat B)} $ is defined in (\ref{eq:H2}). Therefore, for every $ (f,g)\in K_1 $, we have that $ g=\hat{f} $ where $ \hat{f} $ is defined in (\ref{eq:hat-f}) and $ {\cal J}^m(\cdot|(f,\hat{f})) $ is given by (\ref{eq:Jm explicit}).
Consider a sequence $ (x_n)_{n \in \mathbb{N}} \subset \displaystyle\bigcup_{(f,g) \in K_1} A_{(f,g)}^L $.Then, for every $ n \in \mathbb{N}, $ there exists $ (f_n,g_n) \in K_1 $ such that $ x_n \in A_{(f_n,g_n)}^L $ (i.e. $ {\cal J}^m(x_n|(f_n,g_n))\leq L $). Then, $ ((f_n,g_n))_{n \in \mathbb{N}} \subset K_1$ and therefore, up to a subsequence, we can suppose that $ (f_n,g_n) \overset{C_0[0,T]^2}{\underset{}{\longrightarrow}} (f,g) \in K_1$, as $ n \to +\infty. $
	Straightforward computations show that there exists a constant $ M>0 $ such that
	$${\cal J}^m(x_n|(f,g))\leq M \quad \mbox{ for every } n \in \mathbb{N}.$$
	Therefore
$ (x_n)_{n \in \mathbb{N}} \subset A_{(f,g)}^M$, where $ A_{(f,g)}^M $ is the compact set defined in (\ref{eq:AL}).  Then, up to a subsequence, we can suppose that $ x_n \overset{C_0[0,T]}{\underset{}{\longrightarrow}} x \in A_{(f,g)}^M$, as $ n \to +\infty$. Furthermore  $ x\in A_{(f,g)}^L $ since, from Proposition \ref{prop:III LDP continuity condition correlated },
	$$ {\cal J}^m(x|(f,g))\leq \displaystyle\liminf_{n \to+\infty}{\cal J}^m(x_n|(f_n,g_n))\leq L. $$
	 Therefore
	$ \displaystyle\bigcup_{(f,g) \in K_1} A_{(f,g)}^L $ is a compact subset of $ C_0[0,T], $ for any $ L\geq 0 $ and for any level set $ K_1 $ of $ I_{(B,\hat{B})}(\cdot,\cdot) $.
\cvd

{\sl Proof of Proposition \ref{prop:Im good}. }
	From the contraction principle,  Proposition \ref{prop:Im good} will be stated if we show that the rate function $ {\cal H}^m((\cdot,\cdot),\cdot), $ defined in (\ref{eq:H^mjoint}), is a good rate function.
	For $ L\geq 0 $  we prove that
	$$\begin{array}{ccl}
	M_L&=&\{((f,g),x)\in C_0[0,T]^2\times C_0[0,T]: {\cal H}^m((f,g),x)\leq L\}\\
	&=& \{((f,g),x)\in\mathscr{H}_{(B,\hat B)}\times  H_0^1[0,T]: I_{(B,\hat{B})}(f,g)+{\cal J}^m(x|(f,g))\leq L\}
	\end{array}$$
	is a compact subset of $  C_0[0,T]^2\times C_0[0,T] $.
	Note that $ M_L $ is a closed subset of $ C_0[0,T] \times C_0[0,T]^2$ since $ {\cal H}^m(\cdot,(\cdot,\cdot)) $ is lower semicontinuous. Set $ K_1=\{(f,g)\in \mathscr{H}_{(B,\hat B)}: I_{(B,\hat{B})}(f,g)\leq L\}. $ It is easy to verify that
	$$  M_L \subset K_1 \times \bigcup_{(f,g) \in K_1}\{x\in C_0[0,T]:{\cal J}^m(x|(f,g))\leq L\}.  $$
	$ K_1$ is compact since it is a level  set of $ I_{(B,\hat{B})}(\cdot,\cdot) $. Then the set on the right hand side is compact from Lemma \ref{lemma:good rate}. Thus $ M_L $ is a  compact set being a closed subset of a compact set. This completes the proof.
\cvd

We summarize the results we have proved for the family $ ((Z_t^{n,m}-x_0)_{t\in[0,T]})_{n \in \mathbb{N}} $ (for $ m\geq 1 $) in the following
theorem. \begin{theorem}\label{th:LDP-approx}
	Suppose $ \sigma $ and $ \mu $  satisfy Assumption \ref{ass:hp-sigma-mu-I}. 	
For every $ m \geq 1 $, a  large deviation principle with the speed $ \varepsilon^{-2}_n $ and the good rate function $ I^m_Z(\cdot) $ given by (\ref{eq:rate-function-logprice-scaled-corr-psi-m})
	holds for the family $ ((Z_t^{n,m}-x_0)_{t \in [0,T]})_{n \in \mathbb{N}} $.\end{theorem}


\subsection{LDP for the log-price processes}\label{sect:correlatedSVM-notid}

In this section we suppose that Assumptions \ref{ass:hp-sigma-mu-I} and \ref{ass:hp-sigma-II} are fulfilled.

 Theorem \ref{th:LDP-approx} provides a LDP for the families $ (Z^{n,m}-x_0)_{n \in \mathbb{N}} $ for every $ m\geq1, $ but our goal is to get a LDP for the family $ (Z^{n}-x_0)_{n \in \mathbb{N}} $.
 Then, we  prove that the sequence of processes $ ((Z^{n,m}-x_0)_{n \in \mathbb{N}})_{m \geq 1} $ is an exponentially good approximation of $ (Z^{n}-x_0)_{n \in \mathbb{N}} $.
 Let us give the definition of exponentially good approximations. The main reference  for this section is \cite{Bax-Jai}.

\begin{definition}
	Let $ (E,d_E) $ be a metric space and for $\delta>0$, define  $ \Gamma_\delta=\{(\tilde{x},x):d_E(\tilde{x},x)>\delta\} \subset E\times E. $  For each $ n\in\N $ and  $ m \in\N $, let $ (\Omega,\mathscr{F}^n,\mathbb{P}^{n,m}) $ be a probability space, and let the $ E $-valued random variables $ Z^n $ and $ Z^{n,m} $ be distributed according to the joint law $ \mathbb{P}^{n,m} $, with marginals $ {\mu}^{n} $ and $ \mu^{n,m} $ respectively. The families $ (Z^{n,m} )_{n\in \N}$ for $m\geq 1$ are called   \textbf{exponentially good approximations} of $( {Z}^{n})_{n\in \N} $ at the speed $\gamma_n$ if, for every $ \delta>0 $, the set $ \{\omega: ({Z}^{n}, Z^{n,m})\in \Gamma_\delta \} $ is $ \mathscr{F}^{n} $-measurable and
	$$  \lim_{m \to +\infty} \limsup_{n\to +\infty}\gamma_n^{-1} \log \mathbb{P}^{n,m}(\Gamma_\delta)= -\infty. $$
	Similarly, the measures $ (\mu^{n,m})_{n\in \N} $ for $m\geq 1$  are exponentially good approximations of $ (\mu^{n})_{n\in \N} $ if one can construct probability spaces $ (\Omega,\mathscr{F}^n,\mathbb{P}^{n,m}) $ as above.
\end{definition}

Next theorem, Theorem 3.11 in \cite{Bax-Jai}, states that under a suitable condition if for each $ m\geq 1 $ the sequence $ (\mu^{n,m})_{n\in \N} $ satisfies a large deviation principle with the rate function $ I^m ,$ then also $ (\mu^{n})_{n\in \N} $  satisfies a large deviation principle with the rate function $ I$,  obtained in terms of the $ I^m $'s.

\begin{theorem}\label{th:exponentially good approximations}{\rm[Theorem 3.11 in \cite{Bax-Jai}]}
	Assume that $( E, \mathscr{B}(E)) $ is a Polish space  and that for each $ m\geq 1, $ $ (\mu^{n,m})_{n\in \N} $ satisfies a LDP with the speed $ \gamma_n $ and the good rate function $ I^m. $ Let $ (\mu^{n})_{n\in \N}$ be a family of probability measures. For every $ \delta>0 $ define
	$$\rho_\delta(\mu^{n,m},\mu^{n})=\inf_{\ep>0}\Big\{ \mu^{n,m}(A)\leq \mu^{n}(A^\delta)+\ep, \quad A \in \mathscr{B}(E) \Big\},$$
	where
	\begin{equation}\label{eq:Adelta} A^\delta=\bigcup_{x \in A}B_\delta(x),\quad {\mbox{and}}\quad B_\delta(x)=\{y \in E: d_E(x,y)<\delta\},\end{equation}
	with $ d_E $ the metric on $ E$. If for every $ \delta>0 $
	\begin{eqnarray}\label{eq:exponentially good approximations BJ}
	 \lim_{m \to +\infty}\limsup_{n\to+\infty}\gamma_n^{-1}\log\rho_\delta(\mu^{n,m},\mu^{n})=-\infty,
	 \end{eqnarray}
	then $(\mu^n)_{n\in \N} $ satisfies a LDP with the speed $ \gamma_n $ and the good rate function $ I $ given by
	$$ I(x)=\underline{I}(x)=\overline{I}(x), $$
	where
	$$ \underline{I}(x)=\lim_{\delta \to 0}\liminf_{m \to +\infty}\inf_{y \in B_\delta(x)}I^m(y), \quad
	 \overline{I}(x)=\lim_{\delta \to 0}\limsup_{m \to +\infty}\inf_{y \in B_\delta(x)}I^m(y). $$
	\end{theorem}

\begin{proposition}\label{prop:identification I}{\rm[Proposition 3.16 in \cite{Bax-Jai}]}
	In the same hypotheses of   Theorem \ref{th:exponentially good approximations}, if
\begin{itemize}
\item $ I^m(x)\overset {m \rightarrow +\infty}{{\longrightarrow}} J(x)$, for   $x\in E$;
\item $ x_m \overset{m \rightarrow +\infty}{{\longrightarrow}}x$ implies $  \liminf_{m \to+\infty}I^m(x_m)\geq J(x),$
\end{itemize}
for some functional $J(\cdot)$, then $ I(\cdot)=J(\cdot)$.

		\end{proposition}
\begin{remark}
	\rm Let $ (E, \mathscr{B}(E)) $ be a Polish space and $ d_E $ the metric on $ E. $ If the random variables $ (Z^{n,m} )_{n\in \N}$ for $m\geq 1$ are exponentially good approximations of  $( {Z}^{n})_{n\in \N} $ at the speed $\gamma_n$, then (\ref{eq:exponentially good approximations BJ}) holds.
	Consider $ A \in \mathscr{B}(E),  $ then
\begin{eqnarray*}  \mu^{n,m}(A)= \mathbb{P}(Z^{n,m} \in A)&=& \mathbb{P}(Z^{n,m} \in A,\, d_E(Z^{n,m},  {Z}^{n})\leq\delta)
 + \mathbb{P}(Z^{n,m} \in A, \, d_E(Z^{n,m},  {Z}^{n})>\delta)\\
 &\leq& \mathbb{P}( {Z}^{n} \in A^\delta)+ \mathbb{P}( d_E(Z^{n,m},  {Z}^{n})>\delta)\\
 &=& \mu^{n}(A^\delta)+\mathbb{P}( d_E(Z^{n,m},  {Z}^{n})>\delta),
\end{eqnarray*}
where $A^\delta$ is defined in (\ref{eq:Adelta}).
It follows that
$$ \rho_\delta(\mu^{n,m},  {\mu}^{n})\leq \mathbb{P}(d_E(Z^{n,m},  {Z}^{n})>\delta)=\mathbb{P}((Z^{n,m},  {Z}^{n}) \in \Gamma_\delta) $$
Therefore, for every $ \delta>0 $
$$\lim_{m \to +\infty}\limsup_{n \to 0}\gamma_n^{-1}\log\rho_\delta(\mu^{n,m}, \mu^{n})=-\infty  $$
	\end{remark}

 \begin{proposition}\label{prop:exponentially good approximations}
 The families $ ((Z^{n,m}) )_{n\in \N}$, $m\geq 1$, defined in  (\ref{eq:logprice-scaled-corr-psi-m})  are  exponentially good approximations, at the speed $\ep_n^{-2}$, of $( {Z}^{n})_{n\in \N} $
 defined in (\ref{eq:logprice-corr-scaled}).
 \end{proposition}
 \begin{proof}
For every $ \delta>0 $ we have to prove that,
\begin{equation}\label{eq:exponentially good approximations}
 \lim_{m \to +\infty}\limsup_{n \to+\infty}\ep_n^{2} \log \mathbb{P}\left(\lVert Z^{n,m} - Z^{n}\rVert_{\infty} > \delta \right)=-\infty,
\end{equation}
that is
$$\lim_{m \to +\infty}\limsup_{n \to+\infty}\ep_n^{2} \log \mathbb{P}\left(\lVert V^n - \rho\Psi_m(\ep_nB,\ep_n\hat{B})\rVert_{\infty} > \delta \right)=-\infty,
$$
where
$ V^n$ and $\Psi_m$ are defined, respectively  in   (\ref{eq:V_n}) and (\ref{eq:psi-m}).
One can easily verify  that in order to prove  equality  (\ref{eq:exponentially good approximations}), it is enough to show that
\begin{eqnarray}\label{eq:exponentially good approximations2}
\lim_{m \to +\infty}\limsup_{n \to+\infty}\ep_n^{2} \log \mathbb{P}\left(\ep_n|\rho| \sup_{t \in [0,T]}\left|  \int_0^t \left[\sigma\left( \ep_n\hat{B}_s\right)-\sigma\Big( \ep_n\hat{B}_{{\lfloor \frac{ms}T\rfloor}\frac Tm}\Big)\right]dB_s\right|  > \delta \right)=-\infty
\end{eqnarray}
Formula  (\ref{eq:exponentially good approximations2}) was established, under  Assumptions \ref{ass:hp-sigma-mu-I} and \ref{ass:hp-sigma-II},
 in  Lemmas 23 and 24 in \cite{Gu1}.
This completes the proof. \cvd
\end{proof}

Then, we are ready to establish a large deviation principle for the family $ (Z^{n}-x_0)_{n \in \mathbb{N}} $.
\begin{theorem}\label{th:IZ}
	Suppose $ \sigma $ and $ \mu $  satisfy Assumptions \ref{ass:hp-sigma-mu-I} and \ref{ass:hp-sigma-II}.
	 Then $ (Z^{n}-x_0)_{n \in \mathbb{N}} $, satisfies a LDP on $ C_0[0,T]$, with the speed $ \ep_n^{-2} $ and the good rate function $ I_Z $ given by
	$$ I_Z(x)=\underline{I}_Z(x)=\overline{I}_Z(x), $$
	where
	$$ \underline{I}_Z(x)=\lim_{\delta \to 0}\liminf_{m \to +\infty}\inf_{y \in B_\delta(x)}I^m_{{Z}}(y), \quad
	{\rm and} \quad
	\overline{I}_Z(x)=\lim_{\delta \to 0}\limsup_{m \to +\infty}\inf_{y \in B_\delta(x)}I^m_{Z}(y) $$
	with $  B_\delta(x)=\{y \in C_0[0,T]: \lVert x-y\rVert_{\infty}<\delta\} $ and $ I_{Z}^m$'s are  defined in (\ref{eq:rate-function-logprice-scaled-corr-psi-m}).
	
\end{theorem}

\proof
For every $ m\geq 1 $, $ (Z^{n,m}-x_0)_{n \in \mathbb{N}} $, where $ Z^{n,m} $ is defined by (\ref{eq:logprice-scaled-corr-psi-m}), satisfies a large deviation principle with the speed $ \ep_n^{-2} $ and the good rate function
 $I^m_{Z}$.
From Proposition  \ref{prop:exponentially good approximations}, $((Z^{n,m}-x_0)_{n \in \mathbb{N}})_{m \geq 1} $ is an exponentially good approximation (at the same speed) of $ (Z^{n}-x_0)_{n \in \mathbb{N}} $. This completes the proof.
	\cvd

\subsection{Identification of the rate function}\label{sect:id-rate}

In this section we suppose Assumptions \ref{ass:hp-sigma-mu-I}, \ref{ass:hp-sigma-II} and \ref{ass:hp-sigma-mu-III} are fulfilled.

Theorem  \ref{th:IZ} provides a LDP for the family $ (Z^n-x_0)_{n \in \mathbb{N}} $ with a rate function $I_Z(\cdot)  $ which is obtained in terms of the $ I_{{Z}}^m(\cdot) $'s, but our goal is to write explicitly the rate function.
Let us define a measurable function $ \Psi:C_0[0,T]^2\to C_0[0,T] $  by
\begin{equation}\label{eq:Psi} \Psi(f,g)(\cdot)= \begin{cases}\displaystyle\int_0^\cdot\sigma(\hat{f}(s))\dot{f}(s)\,ds  &(f,g)\in \cl H_{(B,\hat B)}\\
0& (f,g)\in C_0[0,T]^2/\cl H_{(B,\hat B)}\end{cases}
\end{equation}
The function $ \Psi $ is finite on $ C_0[0,T]^2 $ and, for $f\in H_0^1[0,T]$, $\Psi(f,\hat{f}) $ is differentiable with a square integrable derivative, i.e. $ \Psi(f,\hat{f})\in H_0^1[0,T] .$

\begin{remark}\label{rem:sigma-bounded-positive}
\rm Let $D_{L}$ be defined as in
(\ref{eq:ball-CM}).
Then thanks to  Remarks \ref{rem:continuous}  and   \ref{rem:hat-f}, for $f\in D_L$ there exist constants $\overline{\mu}_L$, $\underline{\sigma}_L$ and $\overline{\sigma}_L$ (depending on $ L$)
such that, for $f\in D_{L}$ and $t\in[0,T]$,  we have
$$|\mu(\hat{f}(t))|\leq \overline{\mu}_L,
\quad \quad
0<\underline{\sigma}_L\leq \sigma(\hat{f}(t))\leq \overline{\sigma}_L.$$
\end{remark}

\begin{remark}\label{rem:mu-sigma-slow}\rm
Thanks to
 Assumptions \ref{ass:hp-sigma-mu-I} and \ref{ass:hp-sigma-mu-III} and Remark \ref{rem:hat-f}, there exist a constant $M>0$
such that, for $f\in H_0^1[0,T]$ and $t\in[0,T]$,  we have
$$|\mu(\hat{f}(t))|+ \sigma(\hat{f}(t))\leq M\lVert f\rVert_{H_0^1[0,T]}^{\alpha}.$$
\end{remark}

Next lemma is a particular case of Lemma 2.13 in \cite{Gu2}.  We give some details of the proof for the sake of completeness.

\begin{lemma}\label{lemma:convergence Psi_m}
	For every $L>0,$ if  $D_{L}$ is the set defined in (\ref{eq:ball-CM}), then one has,
	$$  \lim_{m\to+\infty} \sup_{f \in D_{L}}\lVert\Psi(f,\hat{f})-\Psi_m(f,\hat{f})\rVert_{\infty}=0. $$
	
\end{lemma}
\proof From Lemma 22 in \cite{Gu1}, we have
\begin{equation}\label{eq:Gu1}
\lim_{m\to+\infty}\sup_{f \in D_{L}}\sup_{t \in [0,T]}\bigg|\sigma(\hat{f}(t))-\sigma\Big(\hat{f}\Big({\Big\lfloor \frac{mt}T\Big\rfloor}\frac Tm\Big)\Big)\bigg|=0.\end{equation}
In Remark \ref{rem:psi-m} we showed that, if $ f \in H_0^1[0,T] $, for $ t \in [0,T] $
$$ \Psi_m(f,\hat{f})(t)=\int_0^t \sigma\Big(\hat{f}\Big({\Big\lfloor \frac{ms}T\Big\rfloor}\frac Tm\Big)\Big)\dot{f}(s)ds  $$
	
Then, it is enough  to show that for every $L >0$,
		$$
		\lim_{m \to +\infty}\sup_{f \in D_{L}}\sup_{t \in [0,T]}\bigg|\int_0^t\bigg[\sigma(\hat{f}(s))-\sigma\Big(\hat{f}\Big({\Big\lfloor \frac{ms}T\Big\rfloor}\frac Tm\Big)\Big)\bigg]\dot{f}(s)\,ds\bigg| = 0 .
	$$
		For $f \in H_0^1[0,T] $ and $ m \geq 1 $ we have
		\begin{eqnarray*}
	 \sup_{f \in D_{L}}\sup_{t \in [0,T]}\bigg|\int_0^t\bigg[\sigma(\hat{f}(s))-\sigma\Big(\hat{f}\Big({\Big\lfloor \frac{ms}T\Big\rfloor}\frac Tm\Big)\Big)\bigg]\dot{f}(s)\,ds\bigg|
		 &\leq&  \sup_{f \in D_{L}}\int_0^T\bigg|\sigma(\hat{f}(s))-\sigma\Big(\hat{f}\Big({\Big\lfloor \frac{ms}T\Big\rfloor}\frac Tm\Big)\Big)\bigg||\dot{f}(s)|\,ds\\
	 &\leq&  \sqrt{L T} \sup_{f \in D_{L}}\sup_{t \in [0,T]}\bigg|\sigma(\hat{f}(t))-\sigma\Big(\hat{f}\Big({\Big\lfloor \frac{mt}T\Big\rfloor}\frac Tm\Big)\Big)\bigg|.
		 \end{eqnarray*}
	 Therefore  the claim follows from (\ref{eq:Gu1}). \cvd

\medskip

Now let us introduce the following functional
\begin{equation}\label{eq:limit I^mZ}
\mathcal{I}_{Z}(x)=\begin{cases} \displaystyle \inf_{f \in H_0^1[0,T]}\mathcal{H}((f,\hat{f}),x) & x\in H_0^1[0,T] \\
\phantom{\inf}+\infty &  x\notin H_0^1[0,T],
\end{cases}
\end{equation}
where for every $ f \in H_0^1[0,T] $,
$$\mathcal{H}((f, \hat{f}),x)=\frac12 \lVert f\rVert_{H_0^1[0,T]}^2+\frac12\int_0^T \bigg(\frac{\dot{x}(t)-\mu(\hat{f}(t))-\rho \dot{\Psi}(f,\hat{f})(t)}{\bar{\rho}\sigma(\hat{f}(t))}\bigg)^2\,dt$$
and  $ \Psi$ is defined in (\ref{eq:Psi}).
We shall prove that $ I_Z(\cdot)=\mathcal{I}_Z(\cdot) $.

\begin{remark}
	\rm For $ x\in H_0^1[0,T] $ we have,
		$$\mathcal{I}_Z(x)=\inf_{f \in H_0^1[0,T]}\mathcal{H}((f,\hat{f}),x)\leq \mathcal{H}((0,0),x)=\frac{1}{2\bar{\rho}^2\sigma^2(0)} \int_0^T(\dot{x}(t)-\mu(0))^2\,dt,$$
therefore
		$$\mathcal{I}_Z(x)= \inf_{f \in D_{C_x}}\mathcal{H}((f,\hat{f}),x)  $$
		where $C_x=\frac{1}{2\bar{\rho}^2\sigma^2(0)}\int_0^T(\dot{x}(t)-\mu(0))^2\,dt$ and  $ D_{C_x}=\{f \in H_0^1[0,T]:\lVert f\rVert_{H_0^1[0,T]}^2\leq C_x\} $.
		Similarly, for $ x \in H_0^1[0,T], $
		for every $ m\geq 1,$ we have
 $$ I_{{Z}}^m(x)= \inf_{f \in D_{C_x}}\mathcal{H}_m((f,\hat{f}),x) $$
		 where, we recall,  $ I_{{Z}}^m(\cdot) $ is the rate function defined in (\ref{eq:I^mZ}) and
		$$\mathcal{H}_m((f,\hat{f}),x)=\frac12 \lVert f\rVert_{H_0^1[0,T]}^2+\frac12\int_0^T\bigg( \frac{\dot{x}(t)-\mu(\hat{f}(t))-\rho \dot{\Psi}_m(f,\hat{f})(t)}{\bar{\rho}\sigma(\hat{f}(t))}\bigg)^2dt.$$
	\end{remark}

In order to prove that  $I_Z(\cdot) = \mathcal{I}_Z(\cdot) $, we have to verify  that the hypotheses of Proposition \ref{prop:identification I} are fulfilled. We start by proving the convergence to $ \mathcal{I}_Z(\cdot) $ of the rate functions $ I_{{Z}}^m (\cdot)$'s.
\begin{lemma}\label{lemma:pointwise convergence}
	For every $ x \in C_0[0,T] $ one has
	$$  \lim_{m \to +\infty}I_{{Z}}^m(x)=\mathcal{I}_Z(x), $$
	where $ I_{{Z}}^m(\cdot) $ and $ \mathcal{I}_Z(\cdot) $ are defined in (\ref{eq:rate-function-logprice-scaled-corr-psi-m}) and (\ref{eq:limit I^mZ}), respectively.
\end{lemma}

\proof
 If $x \notin H_0^1[0,T]$, one has   $ I_{{Z}}^m(x)=\mathcal{I}_Z(x)=+\infty$.  If
	 $ x \in H_0^1[0,T]$ we have,
	\begin{eqnarray*}
	|I_{{Z}}^m(x)- \mathcal{I}_Z(x)|&=&\bigg| \inf_{f \in D_{C_x}}\mathcal{H}_m((f,\hat{f}),x)-\inf_{f \in D_{C_x}}\mathcal{H}((f,\hat{f}),x) \bigg|
	\leq \sup_{f \in D_{C_x}}|\mathcal{H}_m((f,\hat{f}),x)-\mathcal{H}((f,\hat{f}),x)|.
	\end{eqnarray*}
	Taking into account Remark \ref{rem:sigma-bounded-positive},  we have
$$\displaylines{
	 \sup_{f \in D_{C_x}}|\mathcal{H}_m((f,\hat{f}),x)-\mathcal{H}((f,\hat{f}),x)|\leq \cr
 \frac{\rho^2}{2\bar{\rho}^2\underline{\sigma}_{C_x}^2}\sup_{f \in D_{C_x}}\bigg[\int_0^T|\dot{\Psi}_m^2(f,\hat{f})(t)-\dot{\Psi}^2(f,\hat{f})(t)|dt
	+ \frac2\rho\int_0^T|\dot{x}(t)||\dot{\Psi}(f,\hat{f})(t)-\dot{\Psi}_m(f,\hat{f})(t)|\,dt\cr
	+ \frac2\rho \int_0^T|\mu(\hat{f}(t))||\dot{\Psi}_m(f,\hat{f})(t)-\dot{\Psi}(f,\hat{f})(t) |dt \bigg].
	}$$
	Now we study the three addends. For the first one we have,
	$$ \sup_{f \in D_{C_x}}\int_0^T|\dot{\Psi}_m^2(f,\hat{f})(t)-\dot{\Psi}^2(f,\hat{f})(t)|dt\leq 2\bar{\sigma}_{C_x}C_x\sup_{f\in D_{C_x}}\sup_{t \in [0,T]}\bigg| \sigma\Big(\hat{f}\Big({\Big\lfloor \frac{mt}T\Big\rfloor}\frac Tm\Big)\Big)-\sigma(\hat{f}(t))\bigg|. $$
From the Cauchy-Schwarz inequality and Remark \ref{rem:sigma-bounded-positive} we have,
	$$\displaylines{\sup_{f \in D_{C_x}}\int_0^T|\dot{x}(t)||\dot{\Psi}(f,\hat{f})(t)-\dot{\Psi}_m(f,\hat{f})(t)|\,dt
\leq
\sqrt{C_x}\lVert x\rVert_{H_0^1[0,T]}|\sup_{f\in D_{C_x}}\sup_{t\in [0,T]}\bigg| \sigma\Big(\hat{f}\Big({\Big\lfloor \frac{mt}T\Big\rfloor}\frac Tm\Big)\Big)-\sigma(\hat{f}(t))\bigg|.}
	$$
Finally,
$$\displaylines{\sup_{f\in D_{C_x}}\int_0^T|\mu(\hat{f}(t))|\,|\dot{\Psi}_m(f,\hat{f})(t)-\dot{\Psi}(f,\hat{f})(t) |dt\leq
	\sqrt{TC_x}\,\bar{\mu}_{C_x}\sup_{f\in D_{C_x}}\sup_{t\in[0,T]}\bigg| \sigma\Big(\hat{f}\Big({\Big\lfloor \frac{mt}T\Big\rfloor}\frac Tm\Big)\Big)-\sigma(\hat{f}(t))\bigg|.}$$

The claim then follows from equation (\ref{eq:Gu1}).
 \cvd

It remains to show that  $x_m\overset{}{\underset{m \to +\infty}{\longrightarrow}}x  $ implies $  \liminf_{m\to+\infty} I_Z^m(x_m)\geq \mathcal{I}_Z(x).$
For this purpose we  need to prove that $ \mathcal{I}_Z(\cdot) $ is lower semicontinuous.
\begin{remark}\label{rem:semicon J}\rm
	For every $f\in H_0^1[0,T] $, we consider the functional
	\begin{equation}\label{eq:H}
	{\mathcal J}(x|(f,\hat{f}))= \begin{cases}
	\displaystyle \frac12\int_0^T\bigg(\frac{\dot{x}(t)-\mu(\hat{f}(t))-\rho\dot{\Psi}(f,\hat{f})(t)}{\bar{\rho}\sigma(\hat{f}(t))} \bigg)^2dt & x\in H_0^1[0,T]\\
	  +\displaystyle\infty &x\notin H_0^1[0,T].
	 \end{cases}
 \end{equation}
It is easy to verify that it is the good rate function of the family  $ ((Z_t^{n, (f,\hat{f})})_{t\in [0,T]})_{n \in \mathbb{N}}, $ where
\begin{equation}\label{eq:Z(f,hatf)} {Z}_t^{n,(f,g)}=  X_t^{n, g} +\rho\Psi(f,g)(t) \quad 0\leq t \leq T,
\end{equation}
 therefore it is lower semicontinuous.
\end{remark}

\begin{remark}\label{rem:B_L compatti}
For $L>0$, denote by $B_{L}$ the level sets in the  space $\cl H_{(B,\hat B)}$, i.e.
\begin{equation}\label{eq:B_L}
B_L=\{(f,g) \in \cl H_{(B,\hat B)}: \lVert f\rVert_{H_0^1[0,T]}^2\leq L\}.
\end{equation}
$ B_L $ is a compact set in $ C_0[0,T]^2$ since $ I_{(B,\hat{B})}(\cdot,\cdot) $ defined in (\ref{eq:couple rate function}) is a good rate function.
\end{remark}
\begin{lemma}\label{lemma: ball-continuity Psi}
	The function $\Psi:C_0[0,T]^2\to C_0[0,T]  $ defined in (\ref{eq:Psi}) is continuous on the set $B_L$ defined in (\ref{eq:B_L}), for every $ L>0.$
\end{lemma}
\proof
	Easily follows from  Lemma \ref{lemma:convergence Psi_m} and the continuity of $\Psi_m$ (for every $m\geq 1$).
	\cvd

In the next lemma we will prove that $ {\mathcal J}(\cdot|(\cdot,\cdot))$ is lower semicontinuous as a function of $ ((f,g),x) \in    B_L\times C_0[0,T]$.
\begin{lemma}\label{lemma:semicontinuity H_2}
	Let $ ((f_n,\hat f_n),x_n) \in   B_L  \times C_0[0,T]$ be a sequence of functions such that
	$((f_n,\hat f_n),x_n) \overset{}{\underset{n \to +\infty}{\longrightarrow}}((f,g),x) $
	in $ C_0[0,T]^2\times C_0[0,T].$
Then one has,
	$$ \liminf_{n\to+\infty}{\mathcal J}(x_n|(f_n,\hat f_n)) \geq {\mathcal J}(x|(f,g))$$
	where $ {\mathcal J}(\cdot|(\cdot,\cdot))$ is defined in (\ref{eq:H}).
	\end{lemma}
\proof
If  $\liminf_{n \to+\infty}{\mathcal J}(x_n|(f_n,\hat{f}_n))=+\infty,$
 there is nothing to prove.
 Therefore we can suppose that $ (x_n)_{n \in \mathbb{N}}\subset H_0^1[0,T] $. Since $ (f_n,\hat f_n) \overset{}{\underset{n \to +\infty}{\longrightarrow}}(f,g) $ in $ C_0[0,T]^2$ we have that $(f,g)\in B_L $, i.e. $g=\hat f$.
Furthermore it is easy  to show that
$$ {\mathcal J}(x_n|(f_n,\hat{f}_n))  \geq  \inf_{t \in [0,T]} \bigg(\frac{\sigma(\hat{f}(t))}{\sigma(\hat{f}_n(t))}\bigg)^2 {\mathcal J}\bigg( x_n+\rho\big(\Psi(f,\hat{f})(\cdot)-\Psi(f_n,\hat{f}_n)(\cdot)\big).
   +  \int_0^{\cdot}(\mu(\hat{f}(s))-\mu(\hat{f}_n(s)))ds\bigg|(f,\hat{f})\bigg) .
$$
Now
$$x_n+\rho\Psi(f,\hat{f})(\cdot)-\Psi(f_n,\hat{f}_n)(\cdot)
   +  \int_0^{\cdot}(\mu(\hat{f}(s))-\mu(\hat{f}_n(s)))ds\overset{C_0[0,T]}{\underset{n \to +\infty}{\longrightarrow}}x,  $$
since  $x_n\overset{C_0[0,T]}{\underset{n \to +\infty}{\longrightarrow}}x$,
$ \Psi(f_n,\hat{f}_n)\overset{C_0[0,T]}{\underset{n \to +\infty}{\longrightarrow}} \Psi(f,\hat{f})$ from Lemma \ref{lemma: ball-continuity Psi} and
 $\mu\circ\hat f_n\overset{C_0[0,T]}{\underset{n \to +\infty}{\longrightarrow}}\mu\circ\hat f$ from Remark \ref{rem:continuous} (which implies
$  \int_0^{\cdot}\mu(\hat{f}_n(s))ds\overset{C_0[0,T]}{\underset{n \to +\infty}{\longrightarrow}}\int_0^{\cdot}\mu(\hat{f}(s))ds$).	
   The claim follows from semicontinuity of $ {\mathcal J}(\cdot|(f,\hat{f})) $ (see Remark \ref{rem:semicon J}) and from the uniform convergence  of  $\sigma\circ\hat f_n\to\sigma\circ\hat f$.
\cvd

\begin{lemma}\label{lemma:semicontinuity I^Z}
	The function $ \mathcal{I}_Z(\cdot) ,$ defined in (\ref{eq:limit I^mZ}) is lower semicontinuous.
\end{lemma}
\proof
	 It is  enough to show that, for $L>0$, the set
	$$
	M=\{x\in C_0[0,T]: \mathcal{I}_Z(x)\leq L\}
	 $$
	is  closed.
	Let $ (x_n)_{n \in \mathbb{N}}\subset M $ be a converging sequence of functions,  $ x_n\overset{C_0[0,T]}{\underset{n \to +\infty}{\longrightarrow}} x $. Thanks to the definition of $\mathcal{I}_Z$
we can choose a sequence $ (f_n)_{n \in \mathbb{N}}\subset H_0^1[0,T] $ such that, for every $ n \in \mathbb{N}, $
	\begin{equation}\label{eq:ineq-norm}\frac12 \lVert f_n\rVert_{H_0^1[0,T]}^2+{\mathcal J}(x_n|(f_n,\hat{f}_n))\leq \mathcal{I}_Z(x_n)+ \frac1n \leq L+ \frac1n\leq (L +1).\end{equation}
Therefore $ (f_n,\hat{f}_n)_{n \in \mathbb{N}}\subset B_{2(L+1)} $, where $ B_{2(L+1)} $ is   the  compact set of $ C_0[0,T]^2 $
defined in (\ref{eq:B_L}) and, up to a subsequence, we can suppose that $$(f_n,\hat{f}_n)\overset{C_0[0,T]^2}{\underset{n \to +\infty}{\longrightarrow}}(f,\hat f)\in B_{2(L+1)}.$$
	Now, ${\mathcal J}(\cdot|(\cdot,\cdot))  $ is lower semicontinuous on $ B_{2(L+1)}\times C_0[0,T]$ from Lemma \ref{lemma:semicontinuity H_2}. Then from  inequality (\ref{eq:ineq-norm}) and the lower semicontinuity of the norm,  we have
	\begin{eqnarray*}
	\mathcal{I}_Z(x)&\leq& \frac12 \lVert f\rVert_{H_0^1[0,T]}^2+{\cal J}(x|(f,\hat{f}))
	\leq \liminf_{n \to \infty}\Big(\frac12 \lVert f_n\rVert_{H_0^1[0,T]}^2+{\cal J}(x_n|(f_n,\hat{f}_n))\Big)\leq L.
	\end{eqnarray*}
	Thus $ x \in M $, and  $ M $ is a closed subset of $ C_0[0,T]. $
\cvd

Now, we are ready  prove the final lemma of this section.
\begin{lemma}\label{lemma:liminfI_Z^m}
	If $ x_m {\underset{m \to +\infty}{\longrightarrow}}x  $ in $ C_0[0,T]$, then
	$$  \liminf_{m \to +\infty}I_Z^m(x_m)\geq \mathcal{I}_Z(x) $$
	where $ I_Z^m(\cdot) $ and $\mathcal{I}_Z(\cdot)  $ are defined, respectively, in (\ref{eq:I^mZ}) and (\ref{eq:limit I^mZ}).
	\end{lemma}

\proof
	Suppose $ x_m \overset{C_0[0,T]}{\underset{m \to +\infty}{\longrightarrow}}x $.
	If   there exist $m_0>0$ such that $ (x_m)_{m \geq m_0}\subset C_0[0,T]\setminus H_0^1[0,T] $,   then
	$  \lim_{m \to +\infty}I_Z^m(x_m)= +\infty $
	and there is nothing to prove.

	Otherwise, we can suppose that $ (x_m)_{m \in \mathbb{N}}\subset H_0^1[0,T]. $ Now, there are two possibilities:
	\begin{enumerate}
		\item [\textit {(i)}] $ \sup_{m \geq 1} \lVert x_m\rVert_{H_0^1[0,T]}^2<+\infty;$
		\item [\textit {(ii)}] $ \sup_{m \geq 1} \lVert x_m\rVert_{H_0^1[0,T]}^2=+\infty.$
	\end{enumerate}

\textit {(i)} The sequence $ (x_m)_{m \in \mathbb{N}} $ is bounded in $ H_0^1[0,T] $, therefore there exist a constant $C>0$ (depending on $\sup_{m \geq 1} \lVert x_m\rVert_{H_0^1[0,T]}^2$), such that
$$I_Z^m(x_m)= \inf_{f \in D_{C}}\mathcal{H}_m((f,\hat{f}),x_m), \quad \mathcal{I}_Z(x_m)= \inf_{f \in D_{C}}\mathcal{H}((f,\hat{f}),x_m).$$
Then we have,
$$\begin{array}{ccl}
|I_Z^m(x_m)-\mathcal{I}_Z(x_m)|&=&\big| \inf_{f \in D_C}\mathcal{H}_m((f,\hat{f}),x_m)- \inf_{f \in D_C}\mathcal{H}((f,\hat{f}),x_m) \big|\\
&\leq& \sup_{f \in D_C}|\mathcal{H}_m((f,\hat{f}),x_m)-\mathcal{H}((f,\hat{f}),x_m)|.
\end{array}$$
Using similar computations as in Lemma \ref{lemma:pointwise convergence}, it follows that
$$ \lim_{m \to+\infty}[I_Z^m(x_m)-\mathcal{I}_Z(x_m)]=0.$$

Moreover, since $ \mathcal{I}_Z(\cdot) $ is lower semicontinuous from Lemma \ref{lemma:semicontinuity I^Z}, we have
\begin{eqnarray*}
 \liminf_{m \to+\infty}I_Z^m(x_m)&=&  \liminf_{m \to+\infty}[(I_Z^m(x_m)-\mathcal{I}_Z(x_m))+\mathcal{I}_Z(x_m)]\\
&=&  \liminf_{m \to+\infty}\mathcal{I}_Z(x_m)\geq \mathcal{I}_Z(x).
\end{eqnarray*}

\textit {(ii)} The sequence $ (x_m)_{m \in \mathbb{N}} $ is not bounded in $ H_0^1[0,T]$, therefore  we can suppose that
$\lim_{m \to+\infty}\lVert x_m\rVert_{H_0^1[0,T]}^2=+\infty.$
We will prove,  in this case, that
\begin{equation}\label{eq:limit I_Z^m infinite}
 \lim_{m\to+\infty}I_Z^m(x_m)=+\infty.
\end{equation}
For every $ u>0 $ we have,
\begin{equation}\label{eq:inf-in-out}
\begin{array}{ll}\displaystyle I_Z^m(x_m)&=\min\big\{\inf_{\lVert f\rVert_{H_0^1[0,T]}^2\leq  \lVert x_m\rVert_{H_0^1[0,T]}^{2u}}\mathcal{H}_m((f,\hat{f}),x_m), \inf_{\lVert f\rVert_{H_0^1[0,T]}^2>  \lVert x_m\rVert_{H_0^1[0,T]}^{2u}}\mathcal{H}_m((f,\hat{f}),x_m))\big\}\\
&\displaystyle\geq \min\big\{ \inf_{\lVert f\rVert_{H_0^1[0,T]}^{2}\leq  \lVert x_m\rVert_{H_0^1[0,T]}^{2u}}{\mathcal J}^m(x_m|(f,\hat{f})), \inf_{\lVert f\rVert_{H_0^1[0,T]}^{2}>  \lVert x_m\rVert_{H_0^1[0,T]}^{2u}}I_{(B,\hat{B})}(f,\hat{f})\big\}
\end{array}\end{equation}
where $ I_{(B,\hat{B})}(\cdot,\cdot) $ and $ \mathcal{J}^m(\cdot|(f,\hat{f})) $ are defined,  respectively, in (\ref{eq:couple rate function}) and (\ref{eq:Jm explicit}). Now we consider the two infima in (\ref{eq:inf-in-out}).
For the second one we have,
\begin{equation} \label{eq:inf-out}
\inf_{\lVert f\rVert_{H_0^1[0,T]}^2>  \lVert x_m\rVert_{H_0^1[0,T]}^{2u}}I_{(B,\hat{B})}(f,\hat{f})=\inf_{\lVert f\rVert_{H_0^1[0,T]}^2> \lVert x_m\rVert_{H_0^1[0,T]}^{2u}}\frac12\lVert f\rVert_{H_0^1[0,T]}^2\geq  \frac12\,\lVert x_m\rVert_{H_0^1[0,T]}^{2u}.\end{equation}
From  Assumption \ref{ass:hp-sigma-mu-III} and the Cauchy-Schwarz inequality, we have
\begin{eqnarray*}\label{eq:inequalities J^m}
&&{\cal J}^m(x_m|(f,\hat{f}))\geq  \frac1{2\bar{\rho}^2\,M^2\, \lVert f\rVert_{H_0^1[0,T]}^{2\alpha}}\bigg( \lVert x_m\rVert_{H_0^1[0,T]}^2-2 \int_0^T\dot{x}_m(t)(\mu(\hat{f}(t))+\rho\dot{\Psi}_m(f,\hat{f})(t))\,dt\bigg)\\
&\geq&  \frac1{2\bar{\rho}^2\,M^2\, \lVert f\rVert_{H_0^1[0,T]}^{2\alpha}}( \lVert x_m\rVert_{H_0^1[0,T]}^2-2M\,\sqrt T \lVert f\rVert_{H_0^1[0,T]}^{\alpha }\lVert x_m\rVert_{H_0^1[0,T]}
- 2 \rho\,M\,\lVert f\rVert_{H_0^1[0,T]}^{\alpha +1}\lVert x_m\rVert_{H_0^1[0,T]} ).
\end{eqnarray*}
Now we can choose $ u>0$ such that $ (\alpha+1)u<1$, therefore, since $\lVert f\rVert_{H_0^1[0,T]}^2\leq  \lVert x_m\rVert_{H_0^1[0,T]}^{2u} $, for large $m$ and a suitable $c>0$,
 one has
\begin{equation}\label{eq:inf-in}
\begin{array}{l}
{\mathcal J}^m(x_m|(f,\hat{f}))\geq
  \frac1{2\bar{\rho}^2\,M^2\, \lVert x_m\rVert_{H_0^1[0,T]}^{2\alpha\, u}}( \lVert x_m\rVert_{H_0^1[0,T]}^2-2M\,\sqrt T
  \lVert x_m\rVert_{H_0^1[0,T]}^{1+\alpha\, u}
- 2 \rho\,M\,\lVert x_m\rVert_{H_0^1[0,T]}^{(\alpha +1)u +1} )\\
\phantom{{\mathcal J}^m(x_m|(f,\hat{f}))}\geq c  \lVert x_m\rVert_{H_0^1[0,T]}^{2(1-\alpha u)}.
\end{array}\end{equation}
So  (\ref{eq:limit I_Z^m infinite}) follows from (\ref{eq:inf-in-out}), (\ref{eq:inf-out}), (\ref{eq:inf-in})  and then the proof is complete.
\cvd

We are ready to identify the rate function $ I_Z(\cdot)$ with $ \mathcal{I}_Z(\cdot). $
\begin{theorem}\label{th:identification}
	Let $ I_Z(\cdot) $ be the good rate function given by
	Theorem \ref{th:IZ}. Then,
	$$ I_Z(x)=\mathcal{I}_Z(x) $$
	for every $ x \in C_0[0,T] $, where $ \mathcal{I}_Z(\cdot) $ is given by (\ref{eq:limit I^mZ}).
\end{theorem}
\proof
From Lemma \ref{lemma:pointwise convergence} and Lemma \ref{lemma:liminfI_Z^m} the  hypotheses of  Proposition \ref{prop:identification I} are fulfilled.  This proves the theorem.
\cvd

\begin{remark}\label{rem:onedim}\rm{\bf[One dimensional case] } Consider  the function  $F:C_0[0,T]\to \R$ defined by $F(x)=x_T$.  $F$ is a continuous function. By the contraction principle follows  that the family of random variables
$(Z_T^{n}-x_0)_{n\in \N}$ ($ Z_T^{n} $  defined in (\ref{eq:logprice-corr-scaled})) satisfies a LDP on $ (\R, \mathscr{B}(\R)) $, with the speed $ \ep_n^{-2} $ and the good rate function $ I_{Z_T} (\cdot)$ given by
\begin{eqnarray*}I_{Z_T}(y)&=&\inf_{x\in H_0^1[0,T]: F(x)=y}\,\,\inf_{f \in H_0^1[0,T]}  \Big(\frac12 \lVert f\rVert_{H_0^1[0,T]}^2+\frac12\int_0^T\bigg( \frac{\dot{x}(t)-\mu(\hat{f}(t))-\rho \dot{\Psi}_m(f,\hat{f})(t)}{\bar{\rho}\sigma(\hat{f}(t))}\bigg)^2dt\Big)\\
&=&
\inf_{f\in H_0^1[0,T]}\,\,\inf_{x \in H_0^1[0,T]: F(x)=y}  \Big(\frac12 \lVert f\rVert_{H_0^1[0,T]}^2+\frac12\int_0^T\bigg( \frac{\dot{x}(t)-\mu(\hat{f}(t))-\rho \dot{\Psi}(f,\hat{f})(t)}{\bar{\rho}\sigma(\hat{f}(t))}\bigg)^2dt\Big)\\&=&
\inf_{f\in H_0^1[0,T]}\,\,\inf_{x \in H_0^1[0,T]: F(x)=y}  \Big(\frac12 \lVert f\rVert_{H_0^1[0,T]}^2+{\mathcal J}(x|(f,\hat f)\Big).
\end{eqnarray*}

So we have to calculate,
\begin{equation}\label{eq:inf-onedim}\inf_{x \in H_0^1[0,T]: F(x)=y}  {\mathcal J}(x|(f,\hat f)=\inf_{x \in H_0^1[0,T]: F(x)=y} \frac12\int_0^T\bigg( \frac{\dot{x}(t)-\mu(\hat{f}(t))-\rho \dot{\Psi}(f,\hat{f})(t)}{\bar{\rho}\sigma(\hat{f}(t))}\bigg)^2dt.\end{equation}
Recall  that ${\mathcal J}(\cdot|(f,\hat f)$ is the rate function  	of the family  $ ((Z_t^{n, (f,\hat{f})})_{t\in [0,T]})_{n \in \mathbb{N}} $  ($f\in H_0^1[0,T]$) defined in (\ref{eq:Z(f,hatf)}). This is
 a family of  Gaussian diffusion processes, then we can  apply the known results  for Gaussian processes shown in Section \ref{sect:ldp}. For every $ n \in \mathbb{N} $, $ ((Z_t^{n, (f,\hat{f})})_{t\in [0,T]})_{n \in \mathbb{N}}, $ is a Gaussian process with mean function,
$$ m^{n,f}(t)= \int_0^t\Big(\mu(\hat{f}(s))-\rho \dot{\Psi}(f,\hat{f})(s)\Big)\,ds, \quad t \in [0,T] $$
and covariance function,
$$ k^{n,f}(t,s)=\mbox{Cov}(Z^{n,(f,\hat{f})}_t,Z^{n,(f,\hat{f})}_s)=\displaystyle\bar{\rho}^2\ep_n^2\int_0^{t\wedge s}\sigma(\hat{f}(u))^2\,du=\ep_n^2k^f(t,s),  \quad s,t \in [0,T].
$$
It is not hard to prove that $ (({Z}_t^{n, (f,\hat{f})})_{t \in [0,T]})_{n \in \mathbb{N}} $ satisfies (also) hypotheses of Theorem \ref{th:ldp-gaussian} and then  the family $ (({Z}^{n,(f,\hat{f})}_t)_{t \in [0,T]})_{n \in \mathbb{N}}$ satisfies a LDP on $ C_0[0,T] $ with the inverse speed $ \ep_n^2 $ and the good rate function ${\mathcal J}(\cdot|(f,\hat f)$ given by,
$$ {\mathcal J}(x|(f,\hat f))=\begin{cases}
\displaystyle\frac 12 \Big\lVert x-\int_0^\cdot\big( {\mu(\hat{f}(t))+\rho \dot{\Psi}(f,\hat{f})(t)}\big)\, dt\Big\rVert^2_{\mathscr{H}^f} & x-\int_0^\cdot\big( {\mu(\hat{f}(t))+\rho \dot{\Psi}(f,\hat{f})(t)}\big)\, dt\in \mathscr{H}^f\\
+\infty &x-\int_0^\cdot\big( {\mu(\hat{f}(t))+\rho \dot{\Psi}(f,\hat{f})(t)}\big)\, dt\notin \mathscr{H}^f
\end{cases} $$
where $ \mathscr{H}^f $ and $ \lVert \cdot\rVert_{\mathscr{H}^f} $ denote, respectively, the reproducing kernel Hilbert space and the related norm associated to the covariance function $ k^f. $
The set  set of paths
$$ y(u)=\int_0^T k^f(u,v)d\lambda(v), \quad u \in [0,T], \lambda \in \mathscr{M}[0,T].$$
is dense in $ \mathscr{H}^f $.
Therefore  in the infimum  (\ref{eq:inf-onedim}) we can consider the functions
$$ x(u)-\int_0^u\big( {\mu(\hat{f}(t))+\rho \dot{\Psi}(f,\hat{f})(t)}\big)\, dt=\int_0^T k^f(u,v)d \lambda(v), \quad u \in [0,T], $$
for some $ \lambda \in \mathscr{M}[0,T]$,
with the additional constraint that $ x(T)=y$. Therefore  we have to minimize the functional
$$ \frac12 \displaystyle\int_0^T\int_0^Tk^f(u,v)d\lambda(u)d\lambda(v),
$$
 (with respect to the measure $ \lambda $) with the additional constraint
$$ \int_0^Tk^f(T,v)d\lambda(v)+\int_0^T\big( {\mu(\hat{f}(t))+\rho \dot{\Psi}(f,\hat{f})(t)}\big)\, dt- y=0. $$
By using   the method of Lagrange multipliers the measure $ \lambda $ must satisfy
$$ \int_0^T\int_0^T k^f(u,v)d\lambda(u)d\eta(v)= \beta\displaystyle\int_0^T k^f(T,v)d\eta(v) , $$
i.e.
$$ \displaystyle\int_0^T\bigg( \int_0^T k^f(u,v)d\lambda(u)-\beta k^f(T,v)\bigg) d\eta(v)=0 $$
for every $ \eta \in \mathscr{M}[0,T], $ for some $ \beta \in \mathbb{R}. $
 Since $ \bigg( v\mapsto \int_0^Tk^f(u,v)d\lambda(u)-\beta k^f(T,v)\bigg) $ is a continuous function, it must be
\begin{equation}\label{eq:beta}
\int_0^Tk^f(u,v)d\lambda(u)-\beta k^f(T,v)=0,
\end{equation}
for all $ v \in [0,T] $. Therefore  the solution is
$${\bar\lambda}= \beta\delta_{\{T\}} , $$
$ \delta_{\{T\}} $ standing for the Dirac mass in $T$.
From  equality (\ref{eq:beta}) with $ v=T, $ we find
$$ \beta=\frac{\int_0^T k^f(u,T)d\lambda(u)}{k^f(T,T)}=\frac{y-\int_0^T\big( {\mu(\hat{f}(t))+\rho \dot{\Psi}(f,\hat{f})(t)}\big)\, dt}{k^f(T,T)}. $$
Then
$${\bar\lambda}=\frac{y-\int_0^T\big( {\mu(\hat{f}(t))+\rho \dot{\Psi}(f,\hat{f})(t)}\big)\, dt}{k^f(T,T)}\delta_{\{T\}} , $$
 satisfies the Lagrange multipliers problem, and it is therefore a critique point for the functional we want to minimize. Since it is a strictly convex functional restricted on a linear subspace of $ \mathscr{M}[0,T] $, it is still strictly convex, and thus the critique point $ \bar{\lambda} $ is actually its unique point of minimum. Hence, we have
\begin{eqnarray*}\inf_{x\in H_0^1[0,T]: F(x)=y}\frac12\int_0^T\bigg( \frac{\dot{x}(t)-\mu(\hat{f}(t))-\rho \dot{\Psi}_m(f,\hat{f})(t)}{\bar{\rho}\sigma(\hat{f}(t))}\bigg)^2dt&=&
\frac12 \displaystyle\int_0^T\int_0^Tk^f(u,v)d\bar\lambda(u)d\bar\lambda(v)\\
&=&\frac12 \frac{\Big(y-\int_0^T(\mu(\hat{f}(t))+\rho \dot{\Psi}(f,\hat{f})(t))\,dt\Big)^2}{\int_0^T \bar{\rho}^2\sigma^2(\hat{f}(t)) dt}.\end{eqnarray*}
Therefore, we have
$$I_{Z_T}(y)=\inf_{f\in H_0^1[0,T]} \Big\{\frac12 \lVert f\rVert_{H_0^1[0,T]}^2+\frac12 \frac{\big(y-\int_0^T(\mu(\hat{f}(t))+\rho \dot{\Psi}(f,\hat{f})(t))\,dt\big)^2}{\int_0^T \bar{\rho}^2\sigma^2(\hat{f}(t)) dt}\Big\}.$$
The same result as in \cite{Gu1}.

\end{remark}

\subsection{Asymptotic estimate for the crossing probability}\label{sect:cross-prob}
Here we assume the following dynamics for the asset price process
$$\begin{cases}
dS_t=S_t\sigma(\hat{B}_t)d(\bar{\rho}W_t+\rho B_t), \quad 0\leq t \leq T\\
S_0=1
\end{cases}$$
where the model has been normalized to have  $ S_0=1 $ and $ \mu=0 $.
Then, the unique solution to the previous equation is given by
$$S_t=\exp\bigg\{-\displaystyle\frac12\int_0^t\sigma(\hat{B}_s)^2\,ds+\bar{\rho}\int_0^t\sigma(\hat{B}_s)\,dW_s+\rho\int_0^t\sigma(\hat{B}_s)\,dB_s\bigg\}  $$
for $ 0 \leq t \leq T$.
We observe that the process $ (S_t)_{t \in [0,T]} $ is a strictly positive local martingale, and hence a supermartingale. If we suppose that  $ \sigma $ has sub-linear growth, then $ (S_t)_{t\in [0,T]} $ is a martingale
 (for further details, see Lemma 9 in \cite{Gu1}). Then, in such a case, $ \mathbb{P} $ is a risk-neutral measure.

Let us consider the case of an up-in bond, i.e. an option that pays one unit of num\'{e}raire if the underlying asset reached a given up-barrier $ U>1. $ This is a path-dependent option whose pay-off is
$$ h_{up}^{in}=\mathbf{1}_{\{\sup_{t \in [0,T]} S_t \geq U\}}=\mathbf{1}_{\{\tau^U \leq T\}} ,\quad
	 \quad
	  \tau^U=\inf\{t \in [0,T]: S_t \geq U \} .$$
	Then, the up-in bond pricing functions in $ t=0 $ is defined by
	$$ \mathbb{E}\big[ \mathbf{1}_{\big\{\sup_{t \in [0,T]} S_t \geq U\big\}}\big]=\mathbb{P}\big(\sup_{t \in [0,T]}S_t \geq U \big)=\mathbb{P}( \tau^U \leq T). $$
	As an application of the  results of the previous section we obtain the asymptotic behavior of the small-noise up-in bond pricing function, that is
	$$ P_n=\mathbb{E}\big[ \mathbf{1}_{\big\{\sup_{t \in [0,T]} S_t^n \geq U\big\}}\big]=\mathbb{P}( \tau_n^U \leq T) $$
	where $ (S^n_t)_{t \in [0,T]} $ is the asset price process in the scaled model
$$\begin{cases}dS_t^n=\ep_nS^n_t\sigma(\ep_n\hat{B}_t)d(\bar{\rho}W_t+ \rho B_t) \\ S_0=1 \end{cases}$$
	and
	$$ \tau^U_n=\inf\{t \in [0,T]: S_t^n \geq U \} . $$
	This problem is nothing but the asymptotic estimate of level crossing for the family $ ((S_t^n)_{t \in [0,T]})_{n \in \mathbb{N}} .$ In this case the probability $ P_n $ has a large deviation limit,
	\begin{equation} \lim_{n \to +\infty}\ep_n^2\log(P_n)=-I_U
	\end{equation}
	for some quantity $ I_U>0. $
	We have that
	$$
	\{\tau_n^U\leq T\} =\bigg\{\displaystyle\sup_{t \in [0,T]} S_t^n \geq  U\bigg\}
	= \bigg\{\sup_{t \in [0,T]} Z_t^n-\log U \geq 0\bigg\}
	$$
	where $((Z_t^n)_{t \in [0,T]})_{n \in \mathbb{N}} $ is the family of the log-price processes defined by
	$$Z_t^n=\bigg\{-\displaystyle\frac12\ep_n^2\int_0^t\sigma(\ep_n\hat{B}_s)^2\,ds+\ep_n\bar{\rho}\int_0^t\sigma(\ep_n\hat{B}_s)\,dW_s+\ep_n\rho\int_0^t\sigma(\ep_n\hat{B}_s)\,dB_s\bigg\}   .$$
	We have already shown in Theorems \ref{th:IZ} and \ref{th:identification} that  that $((Z_t^n)_{t \in [0,T]})_{n \in \mathbb{N}} $ satisfies a LDP with
the speed $ \ep_n^{-2} $ and the good rate function
	$$I_{Z}(x)=\begin{cases}
	\displaystyle\inf_{f \in H_0^1[0,T]}\mathcal{H}((f,\hat{f}),x) & x\in H_0^1[0,T] \\
\phantom{\inf}+\infty &  x\notin H_0^1[0,T]
\end{cases}
$$
Then, we have
$$  -\displaystyle\inf_{x \in \mathring{A}}I_Z(x) \leq \liminf_{n \to +\infty}\ep_n^2\log(P_n)\leq \limsup_{n \to +\infty}\ep_n^2\log(P_n)\leq -\inf_{x \in \bar{A}}I_Z(x)$$
where
$$ A=\bar{A}=\bigg\{x \in C_0([0,T]): \displaystyle\sup_{t \in [0,T]} x(t)- \log U \geq 0 \bigg\} \quad
 \mathring{A}=\bigg\{x \in C_0([0,T]): \displaystyle\sup_{t \in [0,T]} x(t)- \log U>0  \bigg\}. $$

It is a simple calculation to show that,
$$\inf_{x \in \mathring{A}}I_Z(x)=\displaystyle\inf_{x \in \bar{A}}I_Z(x),$$
therefore
$$ \displaystyle\lim_{n \to+\infty} \ep_n^2\log(P_n)=-\inf_{x \in A}I_Z(x)=-I_U.$$
 In what follows we will compute the quantity $ I_U $. We have to minimize the  rate function (as in  Remark \ref{rem:onedim}).
Define
$$ A_t=\{x \in C_0[0,T]: x(t)-\log U=0\}, $$
then,
$$ A=\bigcup_{0 \leq t \leq T} A_t. $$
Therefore, with the same notations as in Remark \ref{rem:onedim},
$$\begin{array}{ccl}
\displaystyle\inf_{x \in A}I_Z(x)&=&\displaystyle\inf_{t \in [0,T]}\inf_{x \in A_t}I_Z(x)\\
&=& \displaystyle\inf_{t \in [0,T]}\inf_{x \in A_t}\inf_{f \in H_0^1[0,T]}\{I_{(B,\hat{B})}(f,\hat{f})+ \mathcal{J}(x|(f,\hat{f}))\}\\
&=& \displaystyle\inf_{f \in H_0^1[0,T]}\inf_{t \in [0,T]}\inf_{x \in A_t}\bigg\{\frac12 \lVert f\rVert^2_{H_0^1[0,T]}+ \frac12 \lVert x- \rho\Psi(f,\hat{f})\rVert^2_{\mathscr{H}^f}\bigg\}.
\end{array}$$

The set  set of paths
$$ y(u)=\int_0^T k^f(u,v)d\lambda(v) \quad u \in [0,T], \lambda \in \mathscr{M}[0,T]$$
is dense in $ \mathscr{H}^f $.
Therefore in  the infimum
$$ \displaystyle\inf_{x \in A_t}\bigg\{\frac12 \lVert f\rVert^2_{H_0^1[0,T]}+ \frac12 \lVert x- \rho\Psi(f,\hat{f})\rVert^2_{\mathscr{H}^f}\bigg\} $$
we can consider the functions
$$ x(u)=\rho\Psi(f,\hat{f})(u)+\displaystyle\int_0^T k^f(u,v)d \lambda(v), \quad u \in [0,T], $$
for some $ \lambda \in \mathscr{M}[0,T]$, with the additional constraint that $ x(t)=\log U $.
The solution (calculations are the same as in Remark \ref{rem:onedim})
is $$\bar{\lambda}= \displaystyle\frac{\log U -\rho\Psi(f,\hat{f})(t)}{k^f(t,t)}\delta_{\{t\}} , $$
and then
\begin{eqnarray*} I_U&=&\displaystyle\inf_{x \in A}I_Z(x)=\displaystyle\inf_{f \in H_0^1[0,T]}\inf_{t \in [0,T]}\bigg\{\frac12 \lVert f\rVert^2_{H_0^1[0,T]} + \frac12 \frac{(\log U-\rho\Psi(f,\hat{f})(t))^2}{k^f(t,t)}\bigg\}\\
&=&
\inf_{f \in H_0^1[0,T]}\inf_{t \in [0,T]} \Big(\frac12 \lVert f\rVert_{H_0^1[0,T]}^2+\frac12 \frac{\big(\log U-\int_0^t\rho \dot{\Psi}(f,\hat{f})(u)\,du\big)^2}{\int_0^t \bar{\rho}^2\sigma^2(\hat{f}(u)) du}\Big) \end{eqnarray*}
\begin{remark}\rm In this way we have established a large deviation estimation of the probability that the asset price process $ (S_t)_{t \in [0,T]} $ crosses the upper barrier $ U. $ The same arguments can be applied if we consider a lower or a double barrier.
\end{remark}
%


\section{More general Volterra processes}\label{sect:VolterraGen}
In this section we  extend the  results obtained for the
log-price process $ Z^{n}$ when $\hat B^n=\ep_n \hat B$  to a more general context
in which $(\hat B^n)_{n\in \N}$ is a  family of Volterra processes which satisfies a LDP.
\subsection{Uncorrelated case}
Notice that we never used the fact that $\hat B^n=\ep_n \hat B$, therefore the same results can be easily deduced for more general  families  of Volterra  processes.
We consider a family
of continuous Volterra type Gaussian processes (see Definition \ref{def:Volterra process}) of the form
\begin{equation}\label{eq:Volterra B^n gen} \hat{B}^n_t=\displaystyle\int_0^t K^n(t,s)dB_s \quad 0\leq t \leq T, \end{equation}
where $K^n$ is a suitable kernel. The covariance function of the process $ \hat{B}^n$, for every $ n \in \mathbb{N}, $ is given by
$$ k^n(t,s)=\displaystyle\int_0^{t\wedge s}K^n (t,u)K^n(s,u) \, du \quad \mbox{for } t,s \in [0,T].
$$
Under suitable conditions on  the covariance functions or  on the  kernels,  the hypotheses of   Theorem \ref{th:ldp-gaussian} are satisfied and a LDP holds.

\begin{assumption}\label{ass:cov}
{\rm \bf[Assumptions on the covariance]}

\bf (a) \rm	There exist an infinitesimal function $ \ep_n $ and an asymptotic covariance function $ {k} $ (regular enough to be the covariance function of a continuous centered Gaussian process)
such that
	$$ {k}(t,s)=\displaystyle\lim_{n \to +\infty} \frac{k^n(t,s)}{\ep^2_n} $$
uniformly for $ t,s \in [0,T].$

\bf (b)	\rm There exist constants $ \beta,M>0 $, such that, for every $ n \in \mathbb{N}$
$$ \displaystyle\sup_{s,t \in [0,T], s\neq t} \frac{|k^n(t,t)+k^n(s,s)-2k^n(t,s)|}{\ep^2_n|t-s|^{2\beta}}\leq M. $$
\end{assumption}

.

\begin{assumption}\label{ass:ker} {\rm \bf[Assumptions on the kernel]}

	\bf (a) \rm There exist an infinitesimal function $ \ep_n $ and a kernel $ K $ (regular enough to be the kernel of a continuous Volterra process) such that
\begin{equation} \label{eq:ker-limit}\displaystyle\lim_{n \to +\infty} \frac{K^n(t,s)}{{\ep_n}}=K(t,s) \end{equation}
	uniformly for $ t,\,s \in[0,T]. $
	
	\bf (b) \rm There exist  constants $ C,\beta >0 $ such that one has
	$$ \frac {1}{\ep_n^2} \displaystyle\int_0^T (K^n(t,u)-K^n(s,u))^2 \, du \leq C|t-s|^{2\beta}.$$
\end{assumption}
Assumptions \bf(a) \rm guarantees that (\ref{eq:covlimit}) holds (see, for example, \cite{Gio-Pac}). Assumptions \bf(b) \rm guarantees exponential tightness of the family
 (see, for example, \cite{Mac-Pac})
The following theorem holds.

\begin{theorem}
	Let $ n \in \mathbb{N} ,$ and $ \hat{B}^n $ be a Volterra process as in (\ref{eq:Volterra B^n gen}).
	Suppose the family $ (K^n)_{n \in \mathbb{N}} $ (respectively $(k^n)_{n \in \mathbb{N}}$) satisfies Assumption \ref{ass:ker} (respectively Assumption \ref{ass:cov}).
	Then, the family of Volterra processes $((\hat{B}^n_t)_{t \in [0,T]})_{n \in \mathbb{N}}  $ satisfies a large deviation principle on $ C_0[0,T] $, with the inverse speed $ \ep_n^{-2} $ and the good rate function
	$$ I_{\hat{B}}(f)=\begin{cases}
\frac 12 \lVert f \rVert_{ {\mathscr{H}_{\hat B}}}^2 \quad f\in  {\mathscr{H}_{\hat B}}\\
	+\infty \qquad \quad f\notin  {\mathscr{H}_{\hat B}}
	\end{cases} $$
	where $  {\mathscr{H}_{\hat B}} $ and $ \lVert \cdot \rVert_{ {\mathscr{H}_{\hat B}}} $ denote, respectively, the reproducing kernel Hilbert space and the related norm associated to the covariance function
$$ k(t,s)=\displaystyle\int_0^{t\wedge s}K (t,u)K(s,u) \, du \quad \mbox{for } t,s \in [0,T].
$$
\end{theorem}

In this case under hypotheses of  Theorem \ref{th:main-uncorr}, we have (from Theorem \ref{th:Chaganty})
that 	a large deviation principle with the speed $ \ep^{-2}_n $ and the good rate function
$$ I_X(x)=\begin{cases}\displaystyle \inf_{\varphi\in \mathscr{H}_{\hat B}}\left[\frac12 \lVert \varphi\rVert_{\mathscr{H}_{\hat B}}^2 + \frac12  \int_0^T\Bigg(\frac{\dot{x}(t)-\mu(\varphi(t))}{\sigma(\varphi(t))} \Bigg)^2 \,dt \right]& x \in \mathscr{H}_{\hat B} \\
		\displaystyle +\infty & x \notin \mathscr{H}_{\hat B}
		\end{cases}
		$$
holds for the family $ (X^{n}-x_0)_{n \in \mathbb{N}} $, where for every $ n \in \mathbb{N} $, $ (X_t^{n})_{t \in [0,T]} $ is defined from (\ref{eq:logprice-uncorr-scaled}).

From Remark \ref{rem:LDP Volterra}, if $\hat f(t)=\int_0^t K(t,s)\dot{f}(s)ds$, we have
$$ I_X(x)=\begin{cases}\displaystyle \inf_{f\in H_0^1[0,T]}\left[\frac12 \lVert f\rVert_{H_0^1[0,T]}^2 + \frac12  \int_0^T\Bigg(\frac{\dot{x}(t)-\mu(\hat{f}(t))}{\sigma(\hat{f}(t))} \Bigg)^2 \,dt \right]& x \in H_0^1[0,T] \\
		\displaystyle +\infty & x \notin H_0^1[0,T]
		\end{cases}
		$$

\subsection{Correlated case}
Now we state    a large deviation principle for the couple $ (\ep_n B,\hat{B}^n)_{n \in \mathbb{N}} $.
First observe that $ (\ep_n B,\hat{B}^n) $ is a Gaussian process (for details see for example \cite{For-Zha}) and therefore the following theorem is an application of Theorem 3.4.5 in \cite{Deu-Str}.

\begin{theorem}\label{th:LDP couple-gen}
 Let $ \ep: \mathbb{N} \to \mathbb{R}_+ $ be an infinitesimal function. Suppose Assumption  \ref{ass:ker}  is fulfilled.
	Then $ ((\ep_nB, \hat{B}^n))_{n \in \mathbb{N}} $ satisfies a large deviation principle on $C_0[0,T]^2$ with the speed $ \ep_n^{-2} $ and the good rate function
	$$
		I_{(B,\hat{B})} (f,g)=
		 \begin{cases}
		\displaystyle \frac12  \int_0^T \dot{f}(s)^2 \, ds & (f,g) \in \cl H_{(B,\hat B)}\\
		\displaystyle +\infty & (f,g)\in C_0[0,T]^2\setminus\cl H_{(B,\hat B)}
		\end{cases}
		$$
where
$$\cl H_{(B,\hat B)}=\{(f,g) \in C_0[0,T]^2: f \in H_0^1[0,T], \, g(t)=\int_0^tK(t,u)\dot{f}(u)\,du, \quad 0\leq t \leq T\}$$
and $K$ is defined from equation (\ref{eq:ker-limit}).
\end{theorem}

Notice that we  used the fact that $\hat B^n=\ep_n \hat B$ only in the proof of the Proposition \ref{prop:exponentially good approximations}, in particular in proving
(\ref{eq:exponentially good approximations2}) it has been used that
$$\lim_{m\to +\infty}\limsup_{n\to +\infty} \ep_n^2 \log \P\Big( \sup_{0\leq t\leq T} \ep_n\Big|\hat B_t-\hat B_{{\lfloor \frac{mt}T\rfloor}\frac Tm}\Big|>\delta\Big)=-\infty. $$
This  result is contained in Lemma 24 in \cite{Gu1}.
The same results can be deduced for exponential tight families  of Volterra  processes.

\begin{lemma} \label{lemma:approx-volterra}
If $(\hat B^n)_{n\in\N}$ is an exponentially tight family at the inverse speed $\ep_n^2$, then
for every $\delta>0$
$$\lim_{m\to +\infty}\limsup_{n\to +\infty} \ep_n^2 \log \P\Big( \sup_{0\leq t\leq T} \Big|\hat B^n_t-\hat B^n_{{\lfloor \frac{mt}T\rfloor}\frac Tm}\Big|>\delta\Big)=-\infty. $$
\end{lemma}
\proof
From exponential tightness, for every $R>0$, there exists a compact set $K_R$ (of equi-continuous functions) such that
$\limsup_{n\to +\infty} \ep_n^2 \log \P\Big( \hat B^n\in K_R^c)\leq -R$. Therefore, for every $\delta>0$,  there exists $m_0>0$, such that for every $m>m_0$
$\limsup_{n\to +\infty} \ep_n^2 \log \P\Big( \sup_{|s-t|\leq T/m} \Big|\hat B^n_t-\hat B^n_s\Big|>\delta)\leq -R.$
Since
$$\displaylines{\lim_{m\to +\infty}\limsup_{n\to +\infty} \ep_n^2 \log \P\Big( \sup_{0\leq t\leq T} \Big|\hat B^n_t-\hat B^n_{{\lfloor \frac{mt}T\rfloor}\frac Tm}\Big|>\delta\Big)\leq
\lim_{m\to +\infty}\limsup_{n\to +\infty} \ep_n^2 \log \P\Big( \sup_{|s-t|\leq T/m} \Big|\hat B^n_t-\hat B^n_s\Big|>\delta\Big),}$$
the claim follows.\cvd

Therefore Proposition \ref{prop:exponentially good approximations} holds also in this more general case.
In the hypotheses of  Theorem \ref{th:IZ}, we have (from Theorem \ref{th:Chaganty})
that 	a large deviation principle with the speed $ \ep^{-2}_n $ and the good rate function
$$ I_Z(x)=\begin{cases}\displaystyle \inf_{(\psi,\varphi)\in \cl H_{(B,\hat B)}}\left[\frac12 \lVert (\psi,\varphi)\rVert_{\cl H_{(B,\hat B)}}^2 + \frac12\int_0^T \bigg(\frac{\dot{x}(t)-\mu(\varphi(t))-\rho \dot{\Psi}(\psi,\varphi)(t)}{\bar{\rho}\sigma(\varphi(t))}\bigg)^2\,dt \right]& x \in H_0^1[0,T] \\
		\displaystyle +\infty & x \notin H_0^1[0,T]
		\end{cases}
		$$
holds for the family $ (Z^{n}-x_0)_{n \in \mathbb{N}} $ ($ (Z_t^{n})_{t \in [0,T]} $ is defined from (\ref{eq:logprice-corr-scaled})).
If $\hat f(t)=  \int_0^t  K(t,s)\dot{f}(s)ds$, from Remark \ref{rem:LDP Volterra}, we have
$$ I_Z(x)=\begin{cases}\displaystyle \inf_{f\in H_0^1[0,T]}\left[\frac12 \lVert f\rVert_{H_0^1[0,T]}^2+\frac12\int_0^T \bigg(\frac{\dot{x}(t)-\mu(\hat{f}(t))-\rho \dot{\Psi}(f,\hat{f})(t)}{\bar{\rho}\sigma(\hat{f}(t))}\bigg)^2\,dt\right]& x \in H_0^1[0,T] \\
		\displaystyle +\infty & x \notin H_0^1[0,T]
		\end{cases}
		$$
	
	\begin{example}\label{ex:not-ss}\rm
Consider the sequence of processes $((U^{H,n}_{ t})_{t\in[0,T]})_{n\in \N}=((U^H_{\ep_n t})_{t\in[0,T]})_{n\in \N}$ where $(U^H_t)_{t\in[0,T]}$ is a fractional Ornstein-Uhlenbeck process.
It is not a self similar process, therefore $(U^H_{\ep_n t})_{t\in[0,T]}$ is not equivalent to a scaled process $(\ep_n^\alpha U^H_{ t})_{t\in[0,T]}$.
Thanks to representation (\ref{eq:FOU}) simple calculations show that, if $k_H$ is the covariance function of a fractional Brownian motion, we have
$$\displaylines{k_U(s,t)=\Cov(U^H_t,U^H_s)= \cr k_H(t,s)-a \,e^{-at} \int_0^t e^{au} k_H(s,u) du - a\,e^{-as} \int_0^s e^{av} k_H(t,v) dv + a^2e^{-a(t+s)} \int_0^t \int_0^s e^{au}e^{av} k_H(u,v) du\, dv.}$$
Therefore the  covariance function of the process $U^{H,n}$ is
$k^n_U(t,s)=k_U(\ep_n s,\ep_n t)$. It is straightforward to prove that Assumption 7.1 is verified for the infinitesimal function $\ep_n^H$ and limit covariance $k_H$.

We will obtain a  sample path  large deviation principle for the family  of  processes $ ((Z_t^{n}-x_0)_{t \in [0,T]})_{n \in \mathbb{N}}$ with the speed function $\varepsilon_n^{-2H}$
and the good rate function given by
$$ I_Z(x)=\begin{cases}\displaystyle \inf_{f\in H_0^1[0,T]}\left[\frac12 \lVert f\rVert_{H_0^1[0,T]}^2+\frac12\int_0^T \bigg(\frac{\dot{x}(t)-\mu(\hat{f}(t))-\rho \dot{\Psi}(f,\hat{f})(t)}{\bar{\rho}\sigma(\hat{f}(t))}\bigg)^2\,dt\right]& x \in H_0^1[0,T] \\
		\displaystyle +\infty & x \notin H_0^1[0,T],
		\end{cases}
		$$
where $\hat f(t)=\int_0^t K_H(t,s)\dot{f}(t)dt$ ($K_H$ is the kernel of the fractional Brownian motion defined in (\ref{eqn:kernel fbm})).
\end{example}

\section*{Acknowledgements.}
 The authors wish to thank the Referee for her/his very useful
comments which  allowed us to  improve the paper.

\end{document}